\documentclass[twocolumn, a4paper,10pt,english]{scrartcl}
\pdfoutput=1
\usepackage{babel}
\usepackage{blkarray}% http://ctan.org/pkg/blkarray
\usepackage{color}
\usepackage{colortbl}
\usepackage{float}
\usepackage{graphicx}
\usepackage[latin1]{inputenc}
\usepackage{lscape} % landscape
\usepackage[]{mathtools}
\usepackage{rotating}
\usepackage[normalem]{ulem}
\usepackage{xspace}
\usepackage[all,cmtip]{xy}

\usepackage[T1]{fontenc}
\usepackage{fancyhdr}
\usepackage{geometry}

\usepackage{amsthm,amsmath,amssymb,amsfonts,dsfont}
\usepackage{array}
\usepackage[ruled]{algorithm2e}
\usepackage[noend]{algpseudocode}
\usepackage{enumitem}
\usepackage{framed}
\usepackage{graphics,graphicx}
\usepackage{mathtools}
\usepackage{multirow}
\usepackage{natbib}
\usepackage{subfig}
\usepackage{url}

\usepackage{pdflscape}

\usepackage{booktabs}
\usepackage{siunitx}

\usepackage{forest}

\usepackage{tikz}
\usetikzlibrary{arrows}
\usetikzlibrary{decorations.markings}

\usepackage{tikz-qtree, tikz-qtree-compat}
\usetikzlibrary{shadows, trees}

\newtheorem{theorem}{Theorem}

\newtheorem{corollary}[theorem]{Corollary}
\newtheorem{definition}[theorem]{Definition}

\newtheorem{lemma}[theorem]{Lemma}

\DeclareOldFontCommand{\rm}{\normalfont\rmfamily}{\mathrm}
\DeclareOldFontCommand{\sf}{\normalfont\sffamily}{\mathsf}
\DeclareOldFontCommand{\tt}{\normalfont\ttfamily}{\mathtt}
\DeclareOldFontCommand{\bf}{\normalfont\bfseries}{\mathbf}
\DeclareOldFontCommand{\it}{\normalfont\itshape}{\mathit}
\DeclareOldFontCommand{\sl}{\normalfont\slshape}{\@nomath\sl}
\DeclareOldFontCommand{\sc}{\normalfont\scshape}{\@nomath\sc}

\def\R{\mathbb{R}}

\def\F{\mathbb{F}}

\DeclareMathOperator{\kernel}{Ker}
\DeclareMathOperator{\image}{Im}
\DeclareMathOperator{\ind}{\mathds{1}}

\DeclareMathOperator{\rank}{rank}

\DeclareMathOperator{\st}{\mbox{s.t.}}
\DeclareMathOperator{\supp}{supp}

\DeclareMathOperator{\iter}{iter}
\DeclareMathOperator{\maxiter}{\mbox{\tt MAX\_ITER}}

\DeclareMathOperator{\error}{error}

\DeclareMathOperator{\bd}{bd}

\DeclareMathOperator{\diam}{diam}

\DeclareMathOperator{\ball}{\mathcal{B}}
\DeclareMathOperator{\essential}{\mathsf{Ess}}

\DeclareMathOperator{\positive}{\mathsf{Pos}}
\DeclareMathOperator{\negative}{\mathsf{Neg}}
\DeclareMathOperator{\low}{low_{\partial}}

\DeclareMathOperator{\lowstar}{low^*}
\DeclareMathOperator{\leftcol}{left}

\DeclareMathOperator{\pivots}{\mathsf{Pivots}}
\DeclareMathOperator{\neigh}{\mathcal{N}}

\DeclareMathOperator{\nnz}{nnz}
\DeclareMathOperator{\paired}{\mathsf{Paired}}
\DeclareMathOperator{\lowerbound}{maxCollision}

\newcommand{\bdmap}[1]{d_{#1}}

\newcommand{\manifold}{\mathcal{M}}
\newcommand{\manif}{\mathbf{M}}

\DeclareMathOperator{\betti}{Betti}

\fancypagestyle{firstpage}{
\lhead{{\Large\bfseries\sffamily Parallel multi-scale reduction of persistent homology filtrations}\\\large Rodrigo Mendoza-Smith and Jared Tanner \\ {\tt \{mendozasmith,tanner\}@maths.ox.ac.uk}}

\rfoot{}
\lfoot{}
\cfoot{\thepage}
}

\graphicspath{{./figures/}}
\begin{document}
\thispagestyle{firstpage}

%%%%%%%%%%%%%%%%%%%%%%%%%%%%%%%%%%%%%%%%%%%%%%%%%%%%%%%%%%%%%%%%%%%%%%%%%%%
\begin{abstract}
The persistent homology pipeline includes the reduction of a,
so-called, boundary matrix.  We extend the work of
\cite{bauer2014clear, chen2011persistent} where they show how to use
dependencies in the boundary matrix to adapt the reduction algorithm presented in
\cite{edelsbrunner2002topological} in such a way as to reduce its
computational cost.  Herein we present a number of additional
dependencies in the boundary matrices and propose a novel parallel
algorithm for the reduction of boundary matrices.  In particular, we
show: that part of the reduction is immediately apparent, give bounds
on the reduction needed for remaining columns, and from these give a
framework for which the boundary reduction process can be
massively parallelised. Simulations on four synthetic examples show
that the computational burden can be conducted in approximately a
thousandth the number of iterations needed by traditional methods.
Moreover, whereas the traditional
boundary reductions reveal barcodes sequentially from a filtration
order, this approach gives an alternative method by which barcodes 
are partly revealed for multiple scales simultaneously and further
refined as the algorithm progresses; simulations show that for a
Vietoris-Rips filtration with $\sim10^4$ simplices, an estimate of
the essential simplices with 95\% precision can be computed
in {\em two iterations} and that the reduction completed to within 1\%
in about ten iterations of our algorithm as opposed to nearly
approximately eight thousand iterations for traditional
methods.
%simulations show that 95\% of the
%essential columns are revealed in {\em two iterations} and the
%reduction completed to within 1\%  in about ten iterations of our algorithm as
%opposed to nearly approximately eight thousand iterations for traditional
%methods.
\end{abstract}

\section{Introduction}

Persistent homology is a technique within topological data analysis, see
\cite{edelsbrunner2002topological, ghrist2014elementary} 
and references therein, that estimates the topological features of a
shape in high-dimensional space from a point-cloud $S \subset \R^d$
sampled from a data manifold $\manifold \subset \R^d$. 
The topological information on the shape $\manifold$ is encoded as a set of {\em Betti numbers},
\begin{equation}
\label{eq:betti_numbers}
\betti(\manifold) := \{b_{p,r} : 0 \leq p \leq d, r \in [0, \infty)\}
\end{equation}
which geometrically represent the number of {$p$-dimensional holes} at scale $r \in [0, \infty)$ in a simplicial complex triangulation of $\manifold$.
Knowledge of the homologies persistent in a dataset can aid 
interpretation of the data, see for example
\cite{carlsson2009topology, carlsson2014topological, carlsson2008local, chan2013topology, lum2013extracting, taylor2015topological, nicolau2011topology}.
In the persistent homology paradigm, the set $\betti(\manifold)$ is approximated at multiple scales $\{r_1, \dots, r_T\} \subset [0, \infty)$ through a two-step procedure.
First, a scale-indexed filtration of simplicial complexes is constructed yielding a simplicial complex $K$ with vertex set $S$ and $m$ simplices.
This simplicial complex is represented as a $m \times m$ matrix $\partial$ defined over the Galois field of two elements $\F_2$ and having $\partial_{i,j} = 1$ iff $\sigma_i \in K$ is a face of $\sigma_j \in K$ with co-dimension $1$.
The matrix $\partial \in \F_2^{m \times m}$, together with the scale $r_{\ell}$ at which each simplex $\sigma_j \in K$ is added to the filtration, encode the necessary information to estimate $\betti(\manifold)$ via persistent homology.
The homologies in the data are then revealed in the second, and final,
step of reducing the boundary matrix $\partial$.  Reduction
algorithms for persistent homology are so named as they can 
essentially be viewed
as acting on each column of the boundary matrix to minimise the
maximum index of its nonzeros while maintaining that the column span 
of the first $j$ columns of $\partial$ remains unchanged for all $j$.
That is, following the notation of \cite{chen2011persistent}, reduction algorithms
entrywise minimise  
\begin{equation}
\label{eq:low}
\low(j) = \left\{\begin{array}{ll}
\max\{i \in [m] : \partial_{i,j} = 1\}& \mbox{if $\partial_j \neq 0$}\\
0 & \mbox{if $\partial_j = 0$,}
\end{array}\right.
\end{equation}
subject to the aforementioned span constraint; we denote the minimum
of $\low(j)$ by $\lowstar(j)$.  
Nonzero values of $\lowstar(j)$ reveal a homology persisting from
$\sigma_{\lowstar(j)}$ to $\sigma_j$.  As $\lowstar(j)$ is a property of the
simplicial complex $K$ \cite{edelsbrunner2010computational} we omit the explicit reference to the
boundary matrix in its notation; moreover, we denote by $\low$ and
$\lowstar$ the vectors of their values in $\{0, 1, \dots, m\}.$
Sec.\ \ref{subsec:simplicial_homology} and
\ref{subsec:simplicial_complex} provides further details on the 
construction of the simplicial complex filtration, the persistent
homology pipeline, and their connections with the underlying data
manifold as they pertain to our main results.

The focus of this manuscript is extending the work 
\cite{bauer2014clear, chen2011persistent} where they show how to use
dependencies in the boundary matrix to adapt the first boundary matrix
reduction algorithm \cite{edelsbrunner2002topological}, restated in
Alg.\ \ref{alg:mat_red}, in such a way as to reduce its
computational cost. 
\begin{small}
\begin{algorithm}[!htbp]
%\begin{framed}
	\KwData{$\partial \in \F_2^{m \times m}$}
	\KwResult{ $\lowstar \in \mathbb{Z}_{m+1}^m$}
	\For{$j \in [m]$}{
		\While{$\exists$ $j_0 < j$ : $\low(j_0) = \low(j)$}{
			$\partial_j \leftarrow \partial_j + \partial_{j_0}$\;
		}
	}
	$\lowstar \leftarrow \low$\;
%\end{framed}
\caption{Standard reduction \cite{edelsbrunner2002topological}}
\label{alg:mat_red}
\end{algorithm}
\end{small}
%For a given boundary matrix $\partial \in \F^{m \times m}$, 

Alg.\ \ref{alg:mat_red} performs the reduction by sequentially 
minimising $\low(j)$ by adding columns $j_0<j$ for which $\low(j_0) =
\low(j)$ until either there is no such column $j_0$ or column $j$ has
been set to zero.
The computational overhead in Alg.\ \ref{alg:mat_red} is in computing
the left-to-right column operations, and has a worst-case complexity
of $\mathcal{O}(m^3)$ which is achieved by an example simplicial
complex in \cite{morozov2005persistence}.  Previous
advances in reduction algorithms primarily focus on decreasing the
computational cost by exploiting structure in $\partial$ to reveal
some entries in $\partial$ can be set to zero
\cite{bauer2014clear, chen2011persistent} 
%are paired with a reduced negative column \cite{chen2011persistent},
%increasing the sparsity in the matrix to reduce the flop count \cite{bauer2014clear},
or by parallelising the reduction by dividing the boundary matrix into
blocks and partially reducing each block \cite{bauer2014clear, bauer2014distributed};
further details of these approaches are
given in Sec.\ \ref{sec:prior_art}.
The aforementioned approaches show substantially improved average
empirical operation count as compared to Alg.\ \ref{alg:mat_red}. 
%, they fail to leverage its full structure.

% Second, the boundary matrix $\partial$ is {\em reduced} into a matrix $\partial^*$ by performing a sequence of {\em left-to-right} column additions; namely, operations of the form $\partial_{j} \leftarrow \partial_{j} + \partial_{\ell}$ for $\ell < j$ and $\partial_k$ being the $k$-th column of $\partial$.
% %
% For $m \in \mathbb{N}$, let $[m]:=\{1, \dots, m\}$. In the standard reduction algorithm shown in in Algorithm \ref{alg:mat_red}, left-to-right column operations are performed until the function $\low: [m] \rightarrow \mathbb{Z}_{m+1}$ defined by
% %

% %
% is injective on its support.
%

%
% The injection $\lowp_{\partial^*}$ is a property of the simplicial complex $K$ rather than the reduced boundary matrix $\partial^*$, so when there is no risk of confusion we denote it by $\lowstar \in \mathbb{Z}_{m+1}^m$.
%
%

%
%$\partial^*$, so that each column of $\partial^*$ is either 
%zero or contains information necessary to represent simplices which
%create and remove homologies.  In its simplest form $\partial^*$ can be
%computed through Alg.\ \ref{alg:mat_red} which simply performs 
%left-to-right column additions so that any column
%$\partial^*_j$ is either: zero if $\partial_j$ was in the span of the
%columns $\partial_{\ell}$ for $\ell<j$, or $\partial^*_j$ is nonzero and
%the largest $i$ such that $\partial^*_{ij}=1$ is minimized.  
%

Our main contribution is by noting both bounds on the values of
$\lowstar(j)$, which are empirically observed to typically identify
a large fraction of the $\lowstar(j)$, and moreover presenting a
framework by which the known $\lowstar(j)$ can be used to to reduce
$\partial$ in a highly parallel fashion.  In addition, the bounds
on $\lowstar(j)$ suggest priorities by which columns might be reduced.
The algorithm is designed to work with some of the previous speedups
as suggested in \cite{bauer2014clear}. 
Moreover, as is discussed in Sec.\ \ref{sec:main_contributions},
%our algorithm distributes the workload on by partitioning the matrix in subsets of columns with indices scattered across the range $[m]$ rather than in contiguous blocks, which
our parallelisation strategy induces an iterative non-local refinement procedure on $\low$ and makes the algorithm fit for early stopping, see Secs.\ \ref{sec:lowstar_convergence} and \ref{sec:essential_convergence}.
%
%acts across the whole on potentially elements across of the algorithm improves scattered of the  Rather than splitting the data in blocks contiguous columns, our algorithm partitions the columns of $\partial$ based on the location of previously identified {\em pivots}; namely, columns $j \in [m]$ for which we know that $\low(j) = \lowstar(j)$.
%
%This is coupled with a preprocessing strategy that can readily identify a large set of pivots without the need of performing left-to-right column operations.
%ponly by inspecting the sparsity pattern of the matrix and without the need 
%When coupled with an initialisation strategy that can readily identify a large set of pivots across all scales, the algorithm naturally partitions the columns into independent effectively by inspection of the sparsity pattern of the of the without the need of 
%
%This amounts to a multi-threading strategy in which each processor independently acts on sets of discontiguous columns rather than blocks of contiguous columns. , that is not restricted to containing are not necessarily partition 
%partitions the columns of the matrices in non-contiguous subsets defined by a previously identified pivot, or column $j \in [m]$in packages composed of contiguous blocks like in previous parallelisation strategies, our algorithm distributes operating in blocks as in prevoius parallelisation strategies 
%
An example of
the a resulting reduction algorithm is given in Alg.\ \ref{alg:alpha_beta} whose details are explained further in Sec.\
\ref{sec:main_contributions} and numerical experiments for its application are shown in
Sec.\ \ref{sec:numerics}. Additional strategies to speed up the algorithm are given
in Secs.\ \ref{sec:optional_addon} and \ref{sec:essential_theory},
while Sec.\ \ref{sec:workload_distribution} suggests further extensions 
for computer environments with substantially fewer
processors than $m$.
%
%The field is $mathbb{Z}_2$ as we do not expect real datasets to exhibit torsion, though examples have been found in \cite{Gunnar Images paper}.

\begin{small}
\begin{algorithm}[!htbp]
	\KwData{$\partial \in \F_2^{m \times m}$; $\maxiter \in [m]$}
	\KwResult{ $\lowstar \in \mathbb{Z}_{m+1}^m$}
  \tcp{Phase 0: Initialisation}
  Build $\beta_j$ according to \eqref{eq:beta}\;
  $\pivots \leftarrow \left\{ j \in [m] : \beta_j = \low(j) > 0\right\}$\;
  \For{$j \in \pivots$}{
    $\partial_{\low(j)} \leftarrow 0$\;
  }
  $\iter \leftarrow 0$\;
  \While{$\low \neq \lowstar$ or $\iter \leq \maxiter$}{
    \tcp{Phase I: Local injections}
    \For{$d \in [\dim(K)]$}{
      %$\observed \leftarrow \low(\pivots \cap K_d)$\;
      $\lowerbound \leftarrow 0$\;
      %$\zeros \leftarrow \left\{\ell \in K_d: \low(\ell) = 0 \right\}$\;
      \For{$j \in \left\{\ell \in K_d: \low(\ell) > 0 \right\} \setminus \pivots$}{
        \If{$\low(j) > \lowerbound$}{
          \eIf{$\low(j) \notin \low([j-1])$}{
        %\eIf{$\low(j) \notin \observed$}{
              $\pivots \leftarrow \pivots \cup \left\{ j\right\}$\;
            }{
              $\lowerbound \leftarrow \low(j)$\;
            }
          }
        %$\observed \leftarrow \observed \cup \{\low(j)\}$\;
      }
%      \For{$j \in \left\{\ell \in K_d: \low(\ell) > 0 \right\} \setminus \pivots$}{
%      %\For{$j \in K_d \setminus \left( \pivots \cup \zeros \right)$}{
%        %\eIf{$\low(j) \notin \observed$}{
%        \eIf{$\low(j) \notin \low([j-1])$}{
%          \If{$\low(j) > \lowerbound$}{
%            $\pivots \leftarrow \pivots \cup \left\{ j\right\}$\;
%          }
%        }{
%          $\lowerbound \leftarrow \max(\lowerbound, \low(j))$\;
%        }
%        %$\observed \leftarrow \observed \cup \{\low(j)\}$\;
%      }
%      $\observed \leftarrow \left\{0\right\}$\;
%      \For{$j \in \left\{\ell \in [m]\setminus \pivots: \dim(\sigma_\ell) = d \right\}$}{
%        \If{$\low(j) > \max \observed$}{
%          %$\pair(\low(j), j)$\;
%          $\pivots \leftarrow \pivots \cup \left\{ j\right\}$\;
%          $\partial_{\low(j)} \leftarrow 0$\;
%        }
%        $\observed \leftarrow \observed \cup \{\low(j)\}$\;
%      }
    }
    \tcp{Phase II: Column reduction}
    \For{$j_0 \in \pivots$}{ 
      $\mathcal{N}(j_0) \leftarrow \left\{ \ell \in [m]\setminus[j_0] : \low(\ell) = \low(j_0)\right\}$\;
      %$i \leftarrow \low(j_0)$\;
      %\For{$j \in \left\{\ell \in [m]\setminus[j_0] : \low(\ell) = i\right\}$}{
      \For{$j \in \mathcal{N}(j_0)$}{
        $\partial_j \leftarrow \partial_j + \partial_{j_0}$\;
        \If{$\low(j) = \beta_j$}{
          %$\pair(\low(j), j)$\;
          $\pivots \leftarrow \pivots \cup \left\{ j\right\}$\;
          $\partial_{\low(j)} \leftarrow 0$\;
        }
      }
    }
    \tcp{Increase iteration}
    $\iter \leftarrow \iter + 1$\;
  }
  $\lowstar \leftarrow \low$\;
\caption{Parallel multi-scale reduction}
\label{alg:alpha_beta}
\end{algorithm}
\end{small}

\section{Background}
\label{sec:background}

In this section we give an overview of the pipeline for computing
persistent homology which includes a short introduction to the
topological and algebraic results on which our results are based, and
a review of the prior art in boundary matrix reduction algorithms.
In what follows we adopt the following notation. If $S$ is a set, we let $2^S$ denote the power set of $S$ and $|S|$ be the cardinality of $S$.
We also borrow notation from Combinatorics and let $[m]:= \{1, \dots, m\}$ for $m \in \mathbb{N}$.
We will use $\mathbb{Z}_{m+1}$ as a shorthand for $[m] \cup \{0\}$.
For a function $f: A \rightarrow B$, we let $f(A) = \left\{f(a) \in B: a \in A\right\}$ and $f(\emptyset) = \emptyset$.
If $\partial \in \F_2^{m \times m}$ is a matrix, we let $\partial_j \in \F_2^m$ be its $j$-th column and $\partial_{i, \cdot}$ be its $i$-th row.
The {\em support} of an $m$-dimensional vector $v$ is defined as $\supp(v) = \left\{ i \in [m]: v_i \neq 0\right\}$.
If $ \partial \in \F_2^{m \times m}$ is a boundary matrix, we let $\nnz(\partial)$ or the number of nonzeros entries in $\partial$.
We reserve the notation $\ind\left\{\cdot\right\}$ for indicator functions that return $1$ if the argument is true and $0$ otherwise.
If $\R^d$ is the $d$-dimensional Euclidean space with $p$-norm $\|\cdot\|_p$ and $S \subset \R^d$, we let $\diam(S):=\max_{x, y \in S}\|x- y\|$ be the {\em diameter} of $S$.
Finally, for any $c \in \R^d$ and $r \geq 0$, we let $\ball_r(c):= \left\{ x \in \R^d: \|x - c\|_p \leq r\right\}$ be the $d$-dimensional $p$-ball of radius $r$ centred at $c$.

\subsection{Simplicial homology}
\label{subsec:simplicial_homology}

If $S$ is a finite set and $\alpha \subset \sigma \subset S$, then $\sigma$ is a {\em simplex} of $S$ and $\alpha$ is a {\em face} of $\sigma$.
A set of simplices $K \subset 2^S$ is a {\em simplicial complex} if $\sigma \in K$ implies that every face of $\sigma$ is also in $K$.
The dimension of the simplex is defined as $\dim(\sigma) = |\sigma| - 1$.
The dimension of a simplicial complex $K$ will be defined as $\dim(K) = \max \left\{\dim(\sigma) : \sigma \in K\right\}$.
It will be convenient to assume that the simplices in $K$ are indexed, so that $K = \left\{\sigma_1, \dots, \sigma_{|K|}\right\}$.
In this case, the set of {\em $p$-simplices} is the set of simplices in $K$ which are indexed by
\begin{equation}\label{eq:Kp}
K_p = \left\{ j \in [|K|]: \dim(\sigma_j) = p\right\}.
\end{equation}
The boundary of a simplex $\sigma$ is its set of faces of co-dimension
one; symbolically we denote this by $\bd(\sigma) = \{ \alpha : \dim(\alpha)=\dim(\sigma)-1\}$.

In the subsequent paragraphs we describe, for completeness, the
algebraic structure of simplicial complexes and how it gives rise to
the homology groups; omitting this paragraph does not limit one's
ability to understand the new algorithms proposed.
The set of {\em $p$-chains} $\mathsf{C}_p(K):= 2^{K_p}$ equipped with the symmetric difference operation is an abelian group with neutral element $\emptyset$.
For each $p \in \mathbb{N}$, the group $\mathsf{C}_p(K)$ is related to $\mathsf{C}_{p-1}(K)$ by a {\em boundary map} $\bdmap{p}: \mathsf{C}_p \rightarrow \mathsf{C}_{p-1}$ defined by $\bdmap{p}(c) = \sum_{\sigma \in c} \bd(\sigma)$.
The kernel and image of $\bdmap{p}$ are geometrically meaningful.
Elements in $\mathsf{Z}_p = \kernel \bdmap{p}$ are called {\em $p$-cycles}, while elements in $\mathsf{B}_p = \image \bdmap{p+1}$ are called {\em $p$-boundaries}.
Moreover,  for all $p$ and for all $c \in \mathsf{C}_p$, $\bdmap{p-1} \circ \bdmap{p}(c) = \emptyset$, so $\mathsf{B}_p \subset \mathsf{Z}_p \subset \mathsf{C}_p$, see \cite{edelsbrunner2010computational}.
The quotient space $\mathsf{H}_p = {\mathsf{Z}_p}/{\mathsf{B}_p}$ is called the {\em $p$-th homology group} and its elements are called {\em homology classes}.
The $p$-th {\em Betti number} is defined as $b_p = \rank \mathsf{H}_p$ and counts the number of $p$-dimensional holes of the simplicial complex $K$.

\subsection{Construction of the simplicial complex}
\label{subsec:simplicial_complex}

%{\bf insert somewhere}IF $S \subset \R^d$, then any simplex $\sigma \subset S$ can be realised geometrically via its convex hull $\conv(\sigma)$.

In persistent homology, a simplicial complex $K$ is built from a point-cloud $S \subset \R^d$ sampled from a data manifold $\manifold \subset \R^d$.
Given $S$, the goal of persistent homology is to estimate the relevant homological features of $\manifold$ at all scales $r \in [0,\infty)$.
To do so, a simplicial complex triangulation is computed from $S$ at all scales $ r \in \{r_1, \dots, r_T\} \subset [0, \infty)$ by letting $K = \emptyset$ and adding a simplex $\sigma \subset S$ to $K$ whenever the points in $\sigma$ are sufficiently close to each other.
The notion of closeness is implied by the relevant scale parameter, and is assessed for each scale $r$ via a monotonic function $f_r: 2^S \rightarrow \{0,1\}$ inducing a filtration
\begin{equation}
\label{eq:filtration}
\manif_0 \subset \manif_1 \subset \cdots \subset \manif_T = K
\end{equation}
of simplicial complexes $\manif_i = \left\{ \sigma \in 2^S : f_{r_i}(\sigma) = 1\right\}$.
For example, the function
\[
f_r(\sigma) = \ind{\left\{\bigcap_{x_0 \in \sigma} \ball_r(x_0)\right\}},
\]
generates the {\em \u{C}ech complex} filtration, while the function
\[
f_r(\sigma) = \ind\left\{\diam(\sigma) \leq 2r\right\}
%f_r(\sigma) = \ind{\left\{\bigcap_{x_0 \in \sigma} \left\{ x : \|x- x_0\| \leq r\right\}\right\}},
\]
generates the {\em Vietoris-Rips complex} filtration.
We shall assume that the filtration \eqref{eq:filtration} has $m$ elements and that the largest set in the filtration is a simplicial complex with simplices given by
\[
K = \left\{\sigma_1, \dots, \sigma_m\right\}.
\]
It is further assumed that the simplices in $K$ are indexed according to a {\em compatible ordering}, meaning that the simplices in $\manif_\ell$ always precede the ones in $K \setminus \manif_\ell$, and that the faces of any given simplex always precede the simplex.

Persistent homology tracks how the homology of the filtration changes at each scale $r_t$ or, equivalently, as new simplices are added to the filtration.
%
%At each scale $r_t$ new simplices are added to the filtration potentially changing the homology of $\manif_t$.
%
Indeed, when adding simplex $\sigma_i$ at scale $r_t$, the homology of $\manif_t$ can change in one of two possible ways \cite{edelsbrunner2010computational}.
\begin{enumerate}
\item A class of dimension $\dim(\sigma_i)$ is created. In this case, $\sigma_i$ is a {\em positive} simplex.
\item A class of dimension $\dim(\sigma_{i})-1$ is destroyed. In this case, $\sigma_i$ is a {\em negative} simplex.
\end{enumerate}
If $\sigma_j$ is a negative simplex, then it destroys the class created by a positive simplex $\sigma_i$ with $i = \lowstar(j) <j$, see Lemma \ref{lemma:paired_simplices}.
This observation induces a natural pairing $(\sigma_i, \sigma_j)$ between a negative simplex $\sigma_j$ and the positive simplex $\sigma_i$ it destroys.
Moreover, it allows us to quantify the lifetime of a particular
homology class in the filtration via its {\em homology persistence}
which is the difference between $r_t$ for $\sigma_j$ and $\sigma_i$.

When $r_T$ is sufficiently small, we might find that some simplices are never destroyed; these simplices represent the homology classes that are {\em persistent} in the filtration up to scale $r_T$, and we called them {\em essential}.
The persistence pairs are computed by representing the filtration \eqref{eq:filtration} as a boundary matrix $\partial \in \F_2^{m \times m}$ and applying the {\em reduction algorithm} \cite{edelsbrunner2002topological} to it.
We discuss this process in Sec.\ \ref{subsec:boundary_matrix}
%
%Formally, we say that $\manifold$ is {\em triangulable} if there is a simplicial complex $\manif$ together with a homeomorphism between $\manifold$ and the geometric realisation of $\manif$ \cite{edelsbrunner2010computational}. The complex $\manif$ that approximates $\manifold$ is called the {\em triangulation}.

\subsection{Boundary matrix reduction}
\label{subsec:boundary_matrix}

If $K = \{\sigma_1, \dots, \sigma_m\}$ is a simplicial complex
constructed as in the previous Sec.\ \ref{subsec:simplicial_complex},
it can be represented with a boundary matrix $\partial \in \F_2^{m
  \times m}$ defined by 
\begin{equation}
\label{def:boundary_matrix}
\partial_{i,j} = \left\{\begin{array}{ll}
1 & \mbox{$\sigma_i$ is a face of $\sigma_j$ of co-dimension 1}\\
0 & \mbox{otherwise}
\end{array}\right.
\end{equation}

The boundary matrix $\partial$ is sparse, binary, and upper-triangular and has associated a function $\low: [m] \rightarrow \mathbb{Z}_{m + 1}$ defined as in \eqref{eq:low}.
The matrix $\partial$ will be said to be {\em reduced} when $\low \in \mathbb{Z}_{m+1}^m$ is an injection over its support, $i.e.$ when $\low(j_1) = \low(j_2) > 0$ implies that $j_1 = j_2$.
We use the notation $\partial^*$ to denote a matrix $\partial$ with injective $\low$ and remark that even though $\partial$ can have several reductions, the injection is a property of the complex $K$ and does not depend on any particular reduction $\partial^*$, see \cite[p.183]{edelsbrunner2010computational}.
When there is no risk of confusion, we let $\lowstar$ be the injection of the boundary matrix under consideration.

The vector $\lowstar$ reveals information about the pairings by virtue of the following Lemma.
\begin{lemma}[Pairing \cite{edelsbrunner2002topological}]
\label{lemma:paired_simplices}
If $\sigma_j$ is a negative simplex, then $\sigma_{\lowstar(j)}$ is a positive simplex.
\end{lemma}

Hence, $\lowstar(\cdot)$ partitions the simplices in $K$ as,
\begin{align*}
\positive &= \{j\in[m] : \lowstar(j) = 0\}\\
\negative &= \{j \in [m] : \lowstar(j) > 0\}\\
\essential &= \{j \in \positive : \nexists \;k\; \st\; \lowstar(k) = j\}
\end{align*}
so $\positive \cap \negative = \emptyset$ and $\essential\subset\positive$.
It will also be convenient to define the set $\paired$ as the union of
the set of negative and associated positive simplices
\begin{equation}
\label{eq:paired}
\paired = \left\{ j \in [m] : \lowstar(j) \in [m]\right\} \cup \left\{ \lowstar(j) \in [m] : j \in [m]\right\}
\end{equation}
The technology used to reduce $\partial$ is known as the {\em reduction algorithm} (Alg.\ \ref{alg:mat_red}) and was first presented in \cite{edelsbrunner2002topological}.
The convergence guarantees of Alg.\ \ref{alg:mat_red} are given in Theorem \ref{th:convergence_matred}.
\begin{theorem}[Convergence of Alg.\ \ref{alg:mat_red} \cite{edelsbrunner2002topological}]
\label{th:convergence_matred}
Alg.\ \ref{alg:mat_red} converges for any boundary matrix $\partial \in \F_2^{m\times m}$ in at most $\mathcal{O}(m^3)$ operations.
%obtained from a filtration that fulfils the compatible order condition.
\end{theorem}
We survey the prior art in boundary matrix reduction algorithms in the
subsequent Sec.\ \ref{sec:prior_art}.

\subsection{Boundary matrix algorithm prior art}
\label{sec:prior_art}

Let $K$ be the simplicial complex corresponding to a filtration over a grid $\{r_1, \dots, r_T\}$ and point-cloud $S \subset \R^d$.
The number of simplices in $K$ is of order $\Omega(2^{|S|})$ as $r_T\rightarrow \diam(S)$, which turns the process unfeasible even for moderately large point-clouds $S$ and scales $r_T$.
Hence, though Alg.\ \ref{alg:mat_red} terminates in $\mathcal{O}(m^3)$ steps, we may find that $m = \mathcal{O}(2^{|S|})$ making the reduction unfeasible.
The main computational overhead when reducing a boundary matrix is in performing left-to-right column operations, so many reduction algorithms have implemented strategies that cut the number of column operations required to fully reduce the matrix.
One such strategy was given in \cite{chen2011persistent,
  bauer2014clear} by reducing columns in blocks of decreasing
dimension and observing that, by Lemma \ref{lemma:paired_simplices},
column $\lowstar(j)$ can be set to zero whenever column $j$ is
reduced.  As $\dim(\sigma_{\lowstar(j)})$ is of one lower dimension
that $\dim(\sigma_j)$, clearing results in all positive non-essential
  columns of dimension less than $\dim(K)$ being set to zero without
  zeroing them through column additions.
Setting a positive column to zero when its corresponding negative pair has been found is often called {\em clearing}, so we adopt this terminology.
Alg.\ \ref{alg:alpha_beta} heavily relies on this clearing strategy so we sketch the pseudocode of \cite{chen2011persistent} in Alg.\ \ref{alg:twist}.
\begin{small}
\begin{algorithm}
	\KwData{$\partial \in \F_2^{m \times m}$}
	\KwResult{ $\lowstar \in \mathbb{Z}_m^{m+1}$}
%	$L \leftarrow 0 \in \mathbb{Z}^m$\;
	\For{$d \in \{\dim K, \dim K - 1, \dots, 1\}$}{
		\For{$j \in K_d$}{
			\While{$\exists j_0 < j : \low(j_0) = \low(j)$}{
				$\partial_j \leftarrow \partial_j + \partial_{j_0}$\;
			}
			\If{$\partial_j \neq 0$}{
				$\partial_{\low(j)} \leftarrow 0$\;
			}
		}
	}
	$\lowstar \leftarrow \low$\;
\caption{Standard reduction with a twist \cite{chen2011persistent}}
\label{alg:twist}
\end{algorithm}
\end{small}

Another strategy to reduce the complexity of Alg.\ \ref{alg:mat_red} is called {\em compression} \citep{bauer2014clear} and consists in deriving analytical guarantees to nullify nonzeros in the boundary matrix without affecting the pairing of the simplices.
This has the objective of saving arithmetic operations and hence reducing the flop count.
However for very large values of $m$, which are typical for large
filtration values where $m$ grows exponentially with $|S|$, it is
necessary to also parallelise these approaches in order to be scalable.
%
%Our assumption is that, apart from cutting down the number of column operations, it is necessary to develop reduction strategies that can naturally distribute the workload over as many processors as possible and in such a way that early stopping produces an accurate estimation of $\lowstar$.
%
The idea of distributing the workload of the reduction algorithm has
already been explored in \citep{bauer2014clear, bauer2014distributed}
where the matrix is partitioned into $b$ blocks of contiguous columns
to be independently reduced in a shared-memory or distributed system.
The numerical simulations in \cite{bauer2014distributed} show that these strategies can indeed bring substantial speed-ups when implemented on a cluster, but also requires the provision of the parameter $b$ as well as a number of design choices for a practical implementation.

Apart from parallelisation, a number of other approaches have been proposed.
For instance, \cite{milosavljevic2011zigzag} adapted the Coppersmith-Winograd algorithm \cite{coppersmith1990matrix} to guarantee reduction in time $\mathcal{O}(m^{2.3755})$.
A set of sequential algorithms that exploits duality of vector spaces was given in \cite{de2011dualities, de2011persistent}, but as pointed out in \cite{otter2015roadmap} is only known to give speed-ups when applied to Vietoris-Rips complexes.
%
%(From NINA: Note that the dual algorithm is known to give a speed-up when one computes PH with the VR complex, but not necessarily for other types of complexes (see also the results of our benchmarking for the vertebra data set in the SI).)
%
Finally, \cite{sheehy2014persistent} projects $\partial$ to a low-dimensional space while controlling the error between the resulting barcodes, but doing so in practice can be as costly as reducing the matrix in its ambient space.

\section{Main contributions}
\label{sec:main_contributions}

In this section we describe the theory behind the parallelisation strategy of Alg.\ \ref{alg:alpha_beta}.
The main algorithmic innovation is a strategy to efficiently
distribute the workload of Alg.\ \ref{alg:mat_red} over
$\mathcal{O}(m)$ processors to progressively entrywise minimize $\low
\in \mathbb{Z}_{m+1}^m$. 
Additionally, Alg.\ \ref{alg:alpha_beta} is designed to take advantage of some structural patterns in the boundary matrix in order to minimise the total number of left-to-right column operations.
The core observation is that, by simple inspection, the rows of the
boundary matrix reveal a subset of nonzero entries in $\partial$ which
cannot be modified in the reduction process, and consequently identify
nonzero lower bounds on $\lowstar(j)$ for a subset of $j\in [m]$.
This observation is captured by Definition \ref{def:left} and Definition \ref{def:beta}.

%%%%%%%%
% Definition
%%%%%%%%
\begin{definition}[$\leftcol(\cdot)$]
\label{def:left}
Let $\partial \in \F_2^{m \times m}$ be a boundary matrix. The $\leftcol$ function is defined as
\begin{equation}
\label{eq:left}
\leftcol(i) = \left\{\begin{array}{ll}
\min \left\{ j \in [m] : \partial_{i,j} = 1\right\} & \partial_{i,\cdot}\neq 0\\
0 & \partial_{i, \cdot} = 0
\end{array}\right.
\end{equation}
\end{definition}

\begin{definition}[$\beta_j$]
\label{def:beta}
Let $\partial \in \F_2^{m \times m}$ be a boundary matrix and
\begin{equation}
\label{eq:leftcolset}
\mathcal{L}_j = \{i \in [m] : \leftcol(i) = j\}.
\end{equation}
Then, for $j \in [m]$ we let
\begin{equation}
\label{eq:beta}
\beta_j = \left\{\begin{array}{ll}
\max \mathcal{L}_j & \mathcal{L}_j \neq \emptyset\\
0 & \mathcal{L}_j = \emptyset
\end{array}
\right.
\end{equation}
\end{definition}
The vector $\beta \in \mathbb{Z}_{m+1}^m$ carries a great deal of information about the nature of each column in the boundary matrix.
Some of its properties are explored in Theorem \ref{th:beta_properties}.

\begin{theorem}[Properties of $\beta_j$]
\label{th:beta_properties}
Let $\partial \in \F_2^{m \times m}$ be a boundary matrix and $\beta$ defined as in \eqref{eq:beta}.
Then, the following statements hold.
\begin{enumerate}[label=\text{P.\arabic*}]
\item \label{it:a} $\beta_j = 0 \Leftrightarrow \nexists i \in [m] \st \leftcol(i) = j$
\item \label{it:b}$\beta_j$ is invariant to left-to-right column operations.
\item \label{it:c} $\{j \in [m] : \beta_j > 0\} \subset \negative$.
\item \label{it:e} $\positive \subset \{j \in [m] : \beta_j = 0\}$.
\item \label{it:d} $\beta_j \leq \lowstar(j) \leq \low(j)$ for all $j \in [m]$.
\item \label{it:f} $\beta$ can be computed in time $\mathcal{O}(\nnz(\partial))$.
\end{enumerate}
\end{theorem}

\begin{proof}

$\;$

\begin{enumerate}[label=\text{P.\arabic*}]
\item This follows directly from the definition of $\beta_j$.
\item Let $i \in [m]$ be a row of $\partial$ and $j = \leftcol(i)$.
  Reduction of $\partial_j$ is achieved by adding to $\partial_j$
  columns $\partial_{\ell}$ for $\ell<j$  By Definition \ref{eq:leftcolset}
  $\partial_{i,\ell}=0$ for $\ell<j$ and hence $\partial_{i,j}=1$
  throughout the reduction.  Consequently the set $\leftcol(i)$
  and values $\beta_j$ are fixed in the reduction procedure.
\item 
%{\em If $\exists$ $i \in [m]$ such that $\leftcol(i) = j > 0$, then $j \in \negative$.}
If $\beta_j > 0$, then exists $i \in [m]$ such that $\leftcol(i) = j >
0$. Then, $\partial_{i,j} = 1$ and there does not exist
$\partial_\ell$ with $\ell < j$ such that $\partial_{i, \ell} =
1$. Hence $\partial_{i,j}$ can not be set to zero in the reduction, 
%zeroed out by left-to-right column operations, 
so $j \in \negative$.
\item
This result follows by taking complements on the result \ref{it:c}.
\item 
%{\em Result on $\beta_j \leq \lowstar(j)$}
By \ref{it:b} the set $\mathcal{L}_j$ in \ref{eq:leftcolset} are
nonzero entries in $\partial_j$ which cannot be modified in the
reduction procedure and consequently $\lowstar(j)\ge\beta_j$.
% By Algorithm \ref{alg:mat_red}, $\partial_j$ is only updated whenever there is an $\ell < j$ such that $i = \low(j) = \low(\ell)$, and if this happens then we must also have that $\low(\ell) = \lowstar(\ell)$. Therefore, the column update $\partial_j \leftarrow \partial_j + \partial_{\ell}$ annihilates the nonzero in $\partial_{i, j}$ and does not introduce any new non-zeros below $\partial_{i,j}$. Hence, $\low(j) \geq \lowstar(j)$.

% To prove that $\beta_j \leq \lowstar(j)$, let $j \in [m]$ and consider the set $\leftcolset{j}$ defined in \eqref{eq:leftcolset}. If $\leftcolset{j} = \emptyset$ so $\beta_j = 0 \leq \lowstar(j)$.
% %
% If $\leftcolset{j} \neq \emptyset$ then there exists $i \in \supp(\partial_j)$ such that $j = \leftcol(i)$. Similarly as in \eqref{it:b}, $\leftcol(\cdot)$ is invariant to left-to-right column operations, so $\lowstar(j) \geq \leftcol(i)$ for all $i \in \leftcolset{j}$. Therefore,
%
\[
\lowstar(j) \geq \max \{i \in \supp(\partial_j) : \leftcol(i) = j\} = \beta_j
\]
\item Note that when creating the boundary matrix, $\beta$ can be constructed with the following procedure,
\begin{small}
\begin{algorithm}[!htbp]
	\KwData{Simplicial complex $K = \left\{\sigma_1, \dots, \sigma_m\right\}$}
	\KwResult{ $\beta_j$ for $j \in [m]$}
    \For{$\sigma_j \in K$}{
      $\beta_j \leftarrow 0$\;
      \For{$\sigma_i \in \bd(\sigma_j)$}{
        \If{$\sigma_i$ has not been visited}{
          $\beta_j \leftarrow \max(\beta_j, i)$\;
          Mark $\sigma_i$ as visited\;
        }
      }
    }
\caption{Building $\beta$}
\label{alg:beta_building}
\end{algorithm}
\end{small}

Alg.\ \ref{alg:beta_building} finishes in time proportional to $\sum_j |\bd(\sigma_j)|$, so it has complexity $\mathcal{O}\left(\nnz(\partial)\right)$.

\end{enumerate}
\end{proof}

The value of Theorem \ref{th:beta_properties} is most immediately
apparent in \ref{it:e} and in particular when $\low(j)=\beta_j$ in
which case $\lowstar(j)$ is determined.  
%that one can compute the vector $\beta$ such that $\lowstar(j) = \low(j)$ whenever $\beta_j = \low(j)$.
%
Our numerical experiments in Sec.\ \ref{sec:numerics} show
empirically that a large number of columns can usually be identified
as already being reduced without need of any column operations.
Moreover, having access to a large number of reduced columns allows a
massive parallelisation as described in Sec.\ \ref{subsec:parallel} and
illustrated in Sec.\ \ref{sec:numerics}.
%to be already reduced, thus permitting us to massively distribute the
%computation across $p$ processors. 
%
While the calculation of $\beta$ adds an initial computation to Alg.\ \ref{alg:alpha_beta}, it can be computed at the matrix-reading cost of $\mathcal{O}(\nnz(\partial))$, so it can often be obtained as a by-product of creating the filtration without penalising the asymptotic computational complexity of constructing the matrix.

Additionally, the vector $\beta$ gives a sufficient condition for a
column to be in $\negative$.
Knowledge that an unreduced column is necessarily negative can be used
in numerous ways.  For example, as discussed in sections
\ref{sec:optional_addon}, \ref{sec:essential_theory}, and
\ref{sec:workload_distribution} this
is useful to produce more reliable estimations of $\lowstar$ in the
case of early-stopping.   Additionally, knowledge that a column is
necessarily negative enhances a notion of clearing as discussed in
Sec.\ \ref{sec:optional_addon}.

Apart from Theorem \ref{th:beta_properties}, Alg.\ \ref{alg:alpha_beta} makes use of the observation that since $\low$ converges to an injection $\lowstar$, one can often inspect the sparsity pattern of $\partial$ to identify regions of $[m]$ where $\low$ is locally an injection.
This local injection property is the basis of the following Theorems
\ref{th:local_injection_i} and \ref{th:local_injection_ii}, which give sufficient
conditions to identify sets $T\subset [m]$ such that $\low(j) =
\lowstar(j)$ for all $j \in T$.
In the remainder of this section it will be convenient to extend some
of our prior notation as follows: for a set $T \subset [m]$, we let
$\low(T) = \left\{ \low(j) : j \in T\right\}$ and $|\low(T)|$ be its
cardinality. 
Informally, the argument in Theorem \ref{th:local_injection_i} is that, for each dimension $d$, any subset of contiguous columns $T \subset K_d \setminus \{j \in K_d : \low(j) = 0\}$ with $\min T = \min K_d$ is already reduced if $\low: T \rightarrow [m]$ is an injection or, equivalently, if $|\low(T)| = |T|$.

% =============
% Local injection I
% =============
\begin{theorem}[Local injection I]
\label{th:local_injection_i}
Let $\partial \in \F_2^{m \times m}$ be the boundary matrix of a
simplicial complex $K$ and let $d \in [\dim(K)]$ and $K_d$ defined as
in \eqref{eq:Kp}.
Let
%Let $\pivots \subset \negative$ be any subset of reduced negative columns and
%
\begin{equation*}
T := K_d \setminus \left\{ j \in K_d : \low(j) = 0\right\} \neq \emptyset
%,\nonumber\\
%\pivots_d &= \left\{ j \in K_d : j \in \pivots \right\}. \nonumber
\end{equation*}
be such that $T = \left\{j_1, \dots, j_{|T|}\right\}$ for $j_1 < \cdots < j_{|T|}$ and for $\ell \leq |T|$, let $T^{\ell} = \left\{j_1, \dots, j_{\ell}\right\}$.
If
\begin{equation}
\label{eq:injection_1}
\left|\low\left(T^{\ell}\right)\right| = \ell,
\end{equation}
then $\low(j) = \lowstar(j)$ for all $j \in T^{\ell}$.
\end{theorem}

\begin{proof}
Let $T$ be defined as in Theorem \ref{th:local_injection_i}.
If $\ell = 1$, then there $\partial_{j_1}$ is the first column of dimension $d$ that appears in the filtration, so there is no column $j \in [j_1]$ that can reduce it.
Hence, suppose that $\ell > 1$ and let $j_t \in T^{\ell}$ be such that $\low(j_t) \neq \lowstar(j_t)$.
Then there is $j < j_t$ such that $\low(j) = \low(j_t)$.
This implies that $j \in K_d \cap [j_t]$ and that $\low(j) > 0$, so $j \in T_{\ell}$.
Hence, $|T_{\ell}| = \ell$, but $|\low(T_{\ell})| < \ell$, which contradicts \eqref{eq:injection_1}.
\end{proof}

Theorem \ref{th:local_injection_ii} addresses a generalisation of Theorem \ref{th:local_injection_i} and argues that for each dimension $d$, any subset of contiguous columns $T \subset K_d \setminus \{j \in K_d : \low(j) = 0\}$ is already reduced if $|\low(T)| = |T|$ and there is no column $j \in K_d \setminus T$ with $j < \min T$ such that $\low(j) \in \low(T)$.

% =============
% Local injection II
% =============
\begin{theorem}[Local injection II]
\label{th:local_injection_ii}
Let $\partial \in \F_2^{m \times m}$ be the boundary matrix of a simplicial complex $K$ and let $d \in [\dim(K)]$.
Define $T$ and $T^{\ell}$ as in Theorem \ref{th:local_injection_i}.
%
%Suppose that 
%%
%\begin{equation}
%\left|\low\left(T^{\ell}\right)\right| < \ell.
%\end{equation}
%
Assume that for $k > 1$ and $k \leq \ell$ it holds that
\begin{equation}
\label{eq:min_max}
\min \low(T^{\ell} \setminus T^{k-1}) > \max \low(T^{k-1})
\end{equation}
and
\begin{equation}
\label{eq:injection_2}
\left|\low\left(\left\{j_k, \dots, j_{\ell}\right\}\right)\right| = \ell - k + 1.
\end{equation}
Then, $\low(j) = \lowstar(j)$ for all $ j \in \left\{j_k, \dots, j_{\ell}\right\}$.
\end{theorem}

\begin{proof}
Let $T$ and $T^{\ell}$ be as in Theorem \ref{th:local_injection_ii} so that $T^{\ell}\setminus T^{k-1} = \left\{j_k, \dots, j_{\ell}\right\}$.
Since $\ell > 0$ and $k > 1$, then $\ell - k + 1 > 0$.
If $k = 2$, then $j_2 \in \low(T^{\ell} \setminus T^{1})$ and invoking \eqref{eq:min_max},
\[
\low(j_2) \geq \min \low(T^{\ell} \setminus T^{1}) > \max \low(T^{1}) = \low(j_1),
\]
so the result follows by invoking Theorem \ref{th:local_injection_i} with $\ell = 2$.
Arguing as in Theorem \ref{th:local_injection_i}, suppose that $k > 2$ and let $j_t \in \left\{j_k, \dots, j_{\ell}\right\}$ be such that $\low(j_t) \neq \lowstar(j_t)$.
Then, there is $j < j_t$ such that $\low(j) = \low(j_t)$.
Given that $j_k \leq j < j_t$ and by \eqref{eq:min_max} then $j \in T^{\ell}\setminus T^{k-1}$.
Hence, $|\low(T_{\ell} \setminus T^{k-1})| < \ell - k + 1$, which contradicts \eqref{eq:injection_2}.
\end{proof}

We now describe how the results presented so far in this section interact in Alg.\ \ref{alg:alpha_beta}.

\subsection{Parallel multi-scale reduction}\label{subsec:parallel}

So far, Theorems \ref{th:beta_properties} -\ref{th:local_injection_ii}
identify a subset of the columns $j$ in $\partial \in \F_2^{m \times m}$ for 
which $\lowstar(j)$ is known without need for performing column
additions.  In order to complete the reduction of $\partial$, that is
identify $\lowstar$, it is necessary to perform column additions
analogous to those in Alg.\ \ref{alg:mat_red}; however,unlike the
sequential ordering in Alg.\ \ref{alg:mat_red}, we will propose massive
parallelisation of the column additions.  Towards this we define the
set of columns $j\in [m]$ for which their $\lowstar(j)$ is known at
any given iteration of a reduction algorithm as the Pivots:
\begin{equation*}
\pivots = \left\{ j \in [m] : \lowstar(j) \in [m] \mbox{ has been
    identified. }\right\} 
\end{equation*}
Moreover, we define the set of column with which a pivot could be
added to in order to reduce their $\low$ as the {\em neighbours} of column
$j$:
\begin{equation}
\label{eq:neighbours}
\neigh(j) = \left\{ \ell >j: \low(\ell)  = \lowstar(j)\right\}.
\end{equation}
%
%Let $\pivots \subset [m]$ be any subset of columns such that $\low(j) = \lowstar(j) > 0$ for all $j \in \pivots$ 
%
Note that as $\lowstar$ are an injection, if $j_1, j_2 \in \pivots$ and $j_1 \neq j_2$, then $\mathcal{N}(j_1) \cap \mathcal{N}(j_2) = \emptyset$ so 
\begin{equation}
\label{eq:partition}
\left(\bigcup_{j \in \pivots} \neigh(j)\right) \bigcup \pivots \subset [m]
\end{equation}
is a partition and each column in $\mathcal{N}(j)$ for $j\in\pivots$
can have their low reduced independently by adding column $j$ to
columns $\ell\in\neigh(j)$.
%
%\begin{equation*}
%[m] = \left(\bigcup_{j \in \pivots} \mathcal{N}(j) \right)\bigcup \left\{\mbox{Other columns}\right\}.
%\end{equation*}
%
In sequential algorithms like Algs.\ \ref{alg:mat_red} and \ref{alg:twist}, the set $\mathcal{N}(j)$ is constructed only after the block $[j-1] \cap K_{\dim(\sigma_j)}$ has been fully reduced.
In contrast, Alg.\ \ref{alg:alpha_beta} starts by inspecting the sparsity pattern of $\partial$ to identify a large set of {\em seed} $\pivots$ that are then used to build \eqref{eq:neighbours} for each $j \in \pivots$.
Given that \eqref{eq:partition} is a partition, each set
$\mathcal{N}(j)$ is partially reduced independently, typically in
parallel and new columns are appended to $\pivots$ whenever
their update is such that their $\lowstar(\ell)$ is identified by  
%$\low(j)$ is identified to be the true $\lowstar(j)$, 
either by Theorems \ref{th:beta_properties} -
\ref{th:local_injection_ii}. 
We elaborate on these stages of Alg.\ \ref{alg:alpha_beta}, and its
convergence, in the remainder of this section.
\subsubsection{Phase 0: Initialising $\pivots$}
The first step of Alg.\ \ref{alg:alpha_beta} is to compute the vector $\beta$, which can be done at cost $\mathcal{O}(\nnz(\partial))$ by Theorem \ref{th:beta_properties}.
This vector is then coupled with $\low$ to build the set of initial pivots
\[
\pivots  \leftarrow \left\{ j \in [m]: \beta_j = \low(j) \right\},
\]
which by Theorem \ref{th:beta_properties}, contains indices for
columns that are already reduced, e.g. $\lowstar(j)$ is known for $j\in\pivots$.
The clearing strategy from Lemma \ref{lemma:paired_simplices} is applied to each of these columns to zero out their corresponding positive pairs.
The resulting set of $\pivots$ is then passed on to the main iteration.

\subsubsection{Phase I: Finding local injections}

In the first phase of the main iteration, the following Alg.\ \ref{alg:pivots_expansion} is applied to each dimension $d \in [\dim(K)]$.
\begin{small}
\begin{algorithm}[!htbp]
	\KwData{$\partial \in \F_2^{m \times m}$; $\pivots \subset \negative$}
	\KwResult{ $\pivots' \subset \negative$ $\st$ $\pivots \subset \pivots'$}
      %$\observed \leftarrow \low(\pivots \cap K_d)$\;
      $\lowerbound \leftarrow 0$\;
      %$\zeros \leftarrow \left\{\ell \in K_d: \low(\ell) = 0 \right\}$\;
      \For{$j \in \left\{\ell \in K_d: \low(\ell) > 0 \right\} \setminus \pivots$}{
        \If{$\low(j) > \lowerbound$}{
          \eIf{$\low(j) \notin \low([j-1])$}{
        %\eIf{$\low(j) \notin \observed$}{
              $\pivots \leftarrow \pivots \cup \left\{ j\right\}$\;
            }{
              $\lowerbound \leftarrow \low(j)$\;
            }
          }
        %$\observed \leftarrow \observed \cup \{\low(j)\}$\;
      }
\caption{Finding local injections}
\label{alg:pivots_expansion}
\end{algorithm}
\end{small}

Alg.\ \ref{alg:pivots_expansion} starts by initialising $\lowerbound \leftarrow 0$ and walking through the non-zero columns in $K_d$ that have not been identified as pivots.
The idea is to identify regions of $K_d$ where Theorems \ref{th:local_injection_i} and \ref{th:local_injection_ii} can be applied to guarantee that a local injection exists.
At each $j$, the variable $\lowerbound$ tracks the largest $i \in \low([j-1]) \cap \low(K_d)$ such that $i = \low(j_1) = \low(j_2)$ for $j_1 \neq j_2 $.
In terms of Theorem \ref{th:local_injection_ii}, $\lowerbound$ plays
the r\^{o}le of $\max \low(T^{j-1})$ in \eqref{eq:min_max}, since if
$\low(j) > \lowerbound$ and $\low(j) \notin \low([j-1])$, then there
is no column $\ell < j$ that can have $\low(\ell) = \low(j)$ and
consequently $\low(j)$ cannot be further reduced and therefore its
$\lowstar(j)$ is known.

\subsubsection{Main iteration II: Parallel column reduction}

Once the first phase has been completed, the sets of $\neigh(j)$ are
computed for each $j\in\pivots$.
% are used to construct the sets
%
%\[
%\mathcal{N}(j) = \left\{\ell > j : \low(\ell) = \low(j)\right\}.
%\]
%
Each column in $\bigcap_{j\in\pivots}\neigh(j)$ has their associated
pivot added to them; that is, 
%processed independently and a
%left-to-right column operation of the form 
$\partial_{\ell} \leftarrow \partial_{\ell} + \partial_j$ is carried out for each $\ell
\in \mathcal{N}(j)$ for $j\in\pivots$. 
Alg.\ \ref{alg:alpha_beta} then marks the column $\ell$ as reduced if $\low(\ell) =
\beta_{\ell}$, $\\ell$ is added to the set of $\pivots$, and 
%
%Whenever a column is identified as reduced, 
the clearing strategy is called to set the associated positive column $\partial_{\lowstar(\ell)}$ to zero.

\subsubsection{Convergence}
The convergence of Alg.\ \ref{alg:alpha_beta} is a consequence of the interplay between Phases I and II of the main iteration.
We describe this mechanism in the following theorem.
\begin{theorem}[Convergence of Alg.\ \ref{alg:alpha_beta}]
Let $\partial \in \F_2^{m \times m}$ be the boundary matrix of a
simplicial complex $K$. Then, Alg.\ \ref{alg:alpha_beta} converges
to $\lowstar$.
\end{theorem}
\begin{proof}
Let $\pivots^{\ell} \subset [m]$ be the set of indices of columns that are known to be already reduced at the start of iteration $\ell$, and let $\low^{\ell}$ be the corresponding estimate of $\lowstar$.
We show that Alg.\ \ref{alg:mat_red} converges by showing that if $\lowstar \neq \low^{\ell}$, then either
\begin{equation}
\label{eq:pivots_increase}
|\pivots^{\ell + 1}| \geq |\pivots^{\ell}| + 1
\end{equation}
or,
\begin{equation}
\label{eq:norm_decrease}
\|\low^{\ell + 1} - \lowstar\| < \|\low^{\ell} - \lowstar\|.
\end{equation}
Note that $|\pivots^{0}| \geq \dim(K) \geq 1$ since at Phase 0 at least the first column of each dimension must be identified as a pivot.

Suppose that at the start of iteration $\ell \geq 0$ the boundary matrix $\partial$ is still not reduced so that $\low^{\ell} \neq \lowstar$.
If a column is identified as reduced at the end of Phase I, then it is marked as a pivot so $|\pivots^{\ell + 1}| \geq |\pivots^{\ell}| + 1$ and we are done.
Hence, assume that no new pivots are identified.
In this case, for each $j \in \pivots^{\ell}$ we construct the set $\mathcal{N}(j)$ as in \eqref{eq:neighbours}.

If $\exists j_0 \in \pivots^{\ell}$ such that $|\mathcal{N}(j_0)| \geq
1$, then Alg.\ \ref{alg:alpha_beta} enters Phase II to reduce columns in $\mathcal{N}(j_0)$.
In this case, \eqref{eq:norm_decrease} holds and we are done.
Otherwise, suppose $|\mathcal{N}(j_0)| = 0$ for all $j \in \pivots^{\ell}$.
In this case, since the matrix is not reduced, the set
\[
\mathcal{C} = \left\{j \in [m] : \left|\mathcal{N}(j)\right| > 0 \right\}
\]
is not empty. Let,
\begin{equation}
\label{eq:j0_minimality}
j_{\min} = \min \mathcal{C}
\end{equation}
and let $d = \dim(\sigma_{j_{\min}})$.
There are two possible cases,
\begin{enumerate}
\item $j_{\min} = \min K_{d}$: In this case, the column must have been
  identified as a pivot at Phase 0, which is a case previously excluded.
\item $j_{\min} > \min K_{d}$: Let
\[
T = K_{d} \setminus \left\{ j \in K_d : \low(j) = 0\right\}.
\]
In this case, it must hold that $|\low([j_{\min}] \cap T)| = |[j_{\min}] \cap T|$ since otherwise, there would exist an $\ell \in [j_{\min}-1] \cap T$ such that $\low(\ell) = \low(j_{\min})$ implying that $|\mathcal{N}(\ell)| > 0$ and thus contradicting the minimality of $j_{\min}$.
Hence, by Theorem \ref{th:local_injection_i}, $\low([j_{\min}] \cap T) = |[j_{\min}] \cap T|$, so $j_{\min}$ must have been identified at Phase I of the iteration.
This is again, a contradiction.
\end{enumerate}
Therefore, it is impossible to have $|\mathcal{N}(j_0)| = 0$ for all $j_0 \in \pivots^{\ell}$ if no pivots were identified at Phase I of the iteration.
%Hence, Phase I of Algorithm \ref{alg:alpha_beta} identifies $j_0$ as part of a local injection and the algorithm can progress.
%
Hence, as long as $\low^{\ell} \neq \lowstar$ either
\eqref{eq:pivots_increase} or \eqref{eq:norm_decrease} hold, so at
each iteration Alg.\ \ref{alg:alpha_beta} will increase the number of
pivots, decrease some entries in $\low$, or both.
\end{proof}

The multi-scale nature of Alg.\ \ref{alg:alpha_beta} is advantageous to implement a number of strategies that allow us progressively refine our knowledge of $\lowstar$.
We describe some of these strategies in the next subsections.

\subsection{Clearing by compression}
\label{sec:optional_addon}

In this section we describe a clearing strategy based on a column
compression arguments inspired by the previous compression arguments
presented in \cite{bauer2014clear} to reduce operation flop-count.
This strategy, in its simplest form presented in Corollary \ref{eq:compression_c8}, can be
straightforwardly implemented in a parallel architecture in Alg.\ 
\ref{alg:alpha_beta}, but is not explicitly listed in Alg.\
\ref{alg:alpha_beta} for conciseness.
%
%Compression arguments have been previously explored 
%
Here, we present a compression result that can in some cases help clear columns in the matrix and in some other cases can give improved bounds on the location of $\lowstar$.
{\em In doing so, it will be convenient to define the set 
$\paired^{\ell} \subset \paired$, as the sets of elements in \eqref{eq:paired} that have been identified at the $\ell$-th iteration by Alg.\ \ref{alg:alpha_beta}.}

%We now present a clearing strategy that partly builds upon the following observations by \cite{bauer2014clear}
%%
%\begin{lemma}
%Let $\partial \in \F_2^{m \times m}$ be a boundary matrix.
%\begin{enumerate}
%\item If $i \in \negative$ and $\partial_j^*$ is such that $\partial_{i,j}^* = 1$ for $j > i$, then $\lowstar(j) > i$.
%\item If $i \in \negative(\partial)$, then $\nexists$ $\ell \in [m]$ $\st$ $i = \lowstar(\ell)$.
%\end{enumerate}
%%\item \citep{bauer2014clear} If $i \in \negative(\partial)$ and $\partial_j^*$ is such that $\partial_{i,j}^* = 1$, then setting $\partial_{i,j}^*\leftarrow 0$ does not affect the pairs.
%\end{lemma}
Note that $\paired^{\ell}$ at iteration $\ell$
include the set of unreduced columns that have been identified as negative.
%by our definition of $\paired$ in \eqref{eq:paired}, the
%set of known paired columns 
%
In particular, by Theorem \ref{th:beta_properties},
\[
\left\{j  \in [m]: \beta_j > 0\right\} \subset \paired^{\ell}.
\]

\begin{lemma}[Clearing by local compression]
\label{lemma:c8}
Let $\partial \in \F_2^{m \times m}$ be a boundary matrix of a simplicial complex $K$. Let $j \in [m]$ and $d = \dim(\sigma_j) \in [\dim(K)]$.
Then at iteration $\ell$ of Alg.\ \ref{alg:alpha_beta},
\begin{equation}
\label{eq:c8_eq}
\lowstar(j) \in L_j:= \left\{ 0 \right\} \cup \left(K_{d-1} \cap [\beta_j, \low(j)]\right) \setminus \paired^{\ell}.
%& \setminus \left\{ \ell \in [m] : \lowstar(\ell) \in [m]\right\} \nonumber \\
%& \setminus \left\{ \lowstar(\ell) \in [m]: \ell \in [m] \right\}. \nonumber
%\left\{ i \in I_{\delta} : i = \lowstar(\ell) > 0, \ell \in [m]\right\} \nonumber\\
%& \setminus \left\{ \ell \in I_{\delta} : \beta_{\ell} > 0\right\} \nonumber\\
%& \setminus \left\{ \ell \in I_{\delta} : \lowstar(\ell) > 0\right\} \nonumber
\end{equation}
\end{lemma}

\begin{proof}
If $j \in [m]$, and $d = \dim(\sigma_j)$, then $\lowstar(j) \in
K_{d-1}$ by the definition of $\partial$; moreover, $\lowstar(j) \in
[\beta_j, \low(j)]$ by Theorem \ref{th:beta_properties}. 
The injection condition on the support of $\lowstar$ implies that if $\lowstar(j)$ can not be equal to any of the observed positive $\lowstar$'s.
Finally, $\lowstar(j)$ can not be equal to the index of any of the identified
negative columns as $\partial_{\lowstar(j)}$ is necessarily positive,
and can not be equal to any of the indices of paired positive columns
by the injection condition and Lemma \ref{lemma:paired_simplices}. 
Hence, $\lowstar(j) \notin \paired^{\ell}$. 
\end{proof}
%
%
%
%so if ths support of $\beta$ is large, one can have a non-trivial estimation of $\essential$ at the beginning of the algorithm.
%
Lemma \ref{lemma:c8} implies the following corollary,
\begin{corollary}
\label{eq:compression_c8}
Let $\partial \in \F_2^{m \times m}$ and let $L_j$ be defined as in \eqref{eq:c8_eq} and such that $|L_j| = 1$.
Let $i \in L_j$.
%I
If $i = 0$, then $\lowstar(j) = 0$. Else, $i = \lowstar(j)$ and
consequently column $j$ is negative and fully reduced and
$\partial_i$ is positive.
\end{corollary}

\subsection{Essential estimation}
\label{sec:essential_theory}

Another way to reuse the partial knowledge of negative and associated
positive columns, as given by $\paired^{\ell} \subset \paired$ be the
set of columns that have been identified as paired at the $\ell$-th
iteration of Alg.\ \ref{alg:alpha_beta}, is to estimate the essential simplices
\begin{lemma}[Essential estimation]
\label{lemma:essential_estimation}
Let $\partial \in \F_2^{m \times m}$ be a boundary matrix and $j \in
[m]$, then at iteration $\ell$
\begin{equation}
\label{eq:essential_estimation}
\essential \subset \left\{[m] \setminus \paired^{\ell}\right\}\setminus \left\{\low(j) : j \in [m]\right\}
\end{equation}
\end{lemma}

\begin{proof}
Clearly, $\essential \cap \paired^{\ell} = \emptyset$ for all $\ell$ by the partition on the columns.
We show that $\essential \cap \left\{\low(j) : j \in [m]\right\} = \emptyset$.
To do this, we follow a proof by contradiction and suppose that
$\low(j) \in \essential$, which implies  
\begin{equation}
\label{eq:branch}
\nexists p \st \lowstar(p) = \low(j)
\end{equation}
as otherwise $p$ is the index of a negative column and $\lowstar(p)$
the index of an associated positive non-essential column.
We show that column $j$ can not be either positive or negative, which leads to a contradiction.
\begin{enumerate}
\item {$j \in \positive$}: As we are determining $\essential \cap
  \left\{\low(j) : j \in [m]\right\}$, we consider only the $j$ such that
  $\low(j) > 0$, but since $j \in \positive$ we must have $\lowstar(j)
  = 0$. By the definition of $j$ being positive there 
%convergence of Algorithm \ref{alg:mat_red} given in Theorem
%\ref{th:convergence_matred} there 
must exist $q < j$ such that $\lowstar(q) = \low(j)$ in order that $j$
can be reduced to the zero column; however, this contradicts
\eqref{eq:branch} by setting $p = q$. 
\item {$j \in \negative$}: By Theorem \ref{th:beta_properties} $\low(j)
  \geq \lowstar(j)$. If $\low(j) = \lowstar(j)$, then
  \eqref{eq:branch} is contradicted with $p = j$. Otherwise, if
  $\low(j) > \lowstar(j)$, then by the definition of $\lowstar(j)$ 
%by convergence of Algorithm \ref{alg:mat_red} given in Theorem
%\ref{th:convergence_matred} 
there must exist $q < j$ such that $\lowstar(q) = \low(j)$ in order to
reduce column $j$. This again, contradicts \eqref{eq:branch}
\end{enumerate}
\end{proof}

Theorem \ref{eq:essential_estimation} implies that the set of essential columns can be iteratively estimated by a set $E^{\ell} \subset [m]$ initialised at $E^0 = [m]$ and iteratively updated as,
\begin{equation}
\label{eq:essential_iteration}
E^{\ell + 1} \leftarrow \left(E^{\ell} \setminus \paired^{\ell}\right) \setminus \left\{\low(j) : j \in [m]\right\}.
\end{equation}

\subsection{Distributing the workload}
\label{sec:workload_distribution}

Central to the efficacy of Alg.\ \ref{alg:alpha_beta} is the typically
large number of pivots available to reduce subsequent columns coupled
with the ability of all column additions in its Phase 2 to be
performed in parallel.  We do not explicitly describe 
parallelisation  strategies for Phase 2 due to architecture specificity,
but make a few remarks regarding communication minimisation and early
termination in Secs.\ \ref{sec:largecardinality} and \ref{sec:termination} respectively.

\subsubsection{Prioritise reduction of sets $\mathcal{N}(j)$ with large cardinality}\label{sec:largecardinality}

Let,
\begin{equation}
\label{eq:neighbours_set}
\mathcal{C} = \left\{j \in \pivots : |\mathcal{N}(j)| > 0\right\},
\end{equation}
be the set of columns that can be used in a given iteration to
decrease $\low$.

% define the data payloads that are passed to each of the $p$ processors in Phase II of Algorithm \ref{alg:alpha_beta}.
% %
% In practice, we might only hope to find an approximation of $\low$, so if $\left| \mathcal{C}\right| \gg p$ it is desireable to optimise our algorithm to only reduce columns that decrease $\low$ as much as possible in {\tt MAX\_ITER} iterations.
% %

If only $p<|\mathcal{C}|$ processors are available, and communication
of these active pivots is the dominant cost, then it may be
desirable to implement only a portion of the possible column
additions available in Phase 2 duration an iteration.  In such a
setting, a priority queue can be used to order 
%accumulate 
the sets $\mathcal{N}(j)$ based on their cardinality and act on
multiples of $p$ sets of active columns per iteration of Phase 2.
%
% At the beginning of the algorithm, form the set $\mathcal{N}(j)$ for each $j \in \pivots$ and append $\left(j, |\mathcal{N}(j)|\right)$ to the queue.
% %
% Then, letting $\mbox{\tt thresh} \in [m]$ be a threshold, we operate on sets $\mathcal{N}(j)$ as long as $|\mathcal{N}(j)| \geq \mbox{\tt thresh}$.
% %
% Then, we take the new pivots and append their indices and cardinality to the queue.
% %
% We then repeat this process until the queue has no sets with cardinality greater than {\tt thresh}.
% %
% At this point, we can either stop or decrease {\tt thresh} and keep going.

\subsubsection{Prioritise reduction of $\mathcal{N}(j) \cap \negative$}\label{sec:termination}

In settings where $\lowstar$ cannot be fully determined due to
computational or time constraints, the reduction ordering can be
prioritised to minimize various objectives.  One such objective is to
minimize $\|\low - \lowstar\|$ in a suitable norm as rapidly as
possible.  Recall, Theorem \ref{th:beta_properties} gives bounds on
the value of $\lowstar$ and $\low$ is available by inspection.  Due to
the the identification of fully reduced negative column, say column
$j$ allows clearing of column $\lowstar(j)$ there is benefit in
prioritising the reduction of columns which are known to be negative,
e.g. $\left\{ j \in [m] : \beta_j > 0\right\} \subset \negative$.
Alternative prioritisations would be application specific, but can
include such information as prioritising values for which the
persistence is greater or improving the estimation of essential columns
as described in Sec.\ \ref{sec:essential_theory}.

\section{Numerical experiments}
\label{sec:numerics}

In this section we evaluate the efficacy of our parallel multi-scale
reduction Alg.\ \ref{alg:alpha_beta} empirically by comparing its
performance against the standard reduction Alg.\ \ref{alg:mat_red} and
standard reduction with a twist Alg.\ \ref{alg:twist}.
Our evaluation is on a set of synthetic simplicial complexes from point-clouds sampled from a set of a predefined set of ensembles.
The point-clouds and their corresponding Rips-Vietoris simplicial complexes are generated with the {\tt Javaplex} \citep{Javaplex} library.
When building the simplicial complex, we supply the following parameters,
\begin{enumerate}
\item Number of points ($N$): The number of points in the point-cloud to be generated.
\item Maximum dimension ($\dim K$): The maximum dimension of the resulting simplicial complex $K$.
\item Maximum filtration value ($r_T$): The maximum radius in a predefined grid of scales $\left\{ r_1, \dots, r_T\right\} \subset [0, \infty)$ used to build the filtration.
\item Number of divisions ($h$): The grid of scales is uniform with
  spacing $h$, that is $h:= r_{i+1} - r_{i}$ for $i \in [T-1]$.
\end{enumerate}
\[
\begin{array}{l | c  c  c  c}
\mbox{ Ensemble } & N & \dim K & r_T & h\\
\hline
\hline
%\mbox{\bf Icosahedron}& \mbox{Yes} & \\
%\mbox{\bf Morozov}& \mbox{No} & \mbox{Order: $7$} \\
\mbox{ Gaussian Points in $\R^3$}             & 15  & 5  & 5  & 10  \\
\mbox{ Figure-8}                              & 15  & 5  & 5  & 10  \\
\mbox{ Trefoil Knot}                          & 15  & 5  & 5  & 10  \\
\mbox{ $\mathbb{S}^2\times \mathbb{S}^2$}     & 15  & 5  & 5  & 10
\end{array}
\]
%
%Implement "parallel twist" and test with example from Nina's paper.

Central to our numerical evaluations in this manuscript is the number
of iterations required.
The reduction in Algs.\ \ref{alg:mat_red} and \ref{alg:twist} is an {\em horizontal} procedure in
the sense that no column in a given dimension is reduced before the preceeding
columns in this dimension have been reduced.
In contrast, Alg.\ \ref{alg:alpha_beta} is a {\em diagonal} procedure as it implements
an horizontal type of reduction in Phases I and II, but rather than completely reducing the columns
in order, it progressively prunes the matrix {\em vertically} by doing left-to-right column operations
every time there is enough information to guarantee that this is possible.
Given this crucial difference in design, it becomes difficult to
find appropriate benchmarking scenarios and, more fundamentally,
to provide a notion of {\em iteration} that is agnostic to the algorithm's architecture.
Hence, in our numerical experiments, we let an iteration be the set of operations that are
performed at each cycle of the outer-most loop of the algorithm.
This is a reasonable convention as it is consistent with each of the algorithm's concurrency and
is also natural for the high-level pseudo-code description of their design.
However, we note that this notion of iteration can be optimistic or pessimistic depending on
the unit of computational overhead that is being measured.
In our experiments, {\em left-to-right} column operations are chosen as the main
unit of computational overhead and, indeed, our algorithm has been designed to minimise
this kind of operations.
Given the embarrassingly parallel nature of Phase II of Alg.\ \ref{alg:alpha_beta},
this model is generous with our algorithm as it does not consider possible limitations to
the number of processors available and it does not account for the communication overhead
between processors.
Finally, as our algorithm's practical performance is greatly limited by the hardware architecture,
we point out that our numerical experiments should serve as a proof-of-concept of a design that can yield
a powerful production implementation rather than a trustworthy comparison of the run-time or flop-count in the general case.

%measureb
%Note, this measure of an iteration is
%generous in that often it counts numerous sequential column additions
%as a single iteration, while at at other times this measure is
%pessimistic as it counts as an iteration checking that there are no
%prior columns which can be added, e.g. determining that a negative
%column to already reduced.
%
%We measure as an iteration of Alg.\
%\ref{alg:alpha_beta} one pass through Phase I and II. 
%This notion of
%an iteration is consistent with the embarrassingly parallel nature of
%Phase II where a numerous pairs of columns can be added
%simultaneously; however the notion is also generous as it does not
%consider possible limitations to the number of processors available
%and includes the 
%sequential Phase I within the same iteration, which we do so as Phase
%1 does not include any column additions.  While the above notions are 
%inconsistent, they are chosen to avoid details of the Phase I search
%cost as compared to that of a column addition, and as a compromise to
%considering an instance of, possibly parallel, column additions as an
%iteration.

The numerical results presented in this section simulate
a parallel implementation and are available at \cite{githubRepo}. 
A truly parallel implementation is being developed.

%The notion of an iteration for differing
%algorithms can be difficult to compare.  We use the notions that
%appear to us most consistent with the implementation of these
%algorithms.  We count as an iteration for Algs.\
%\ref{alg:mat_red} and \ref{alg:twist} the reduction of a
%column within the ``for'' loop.
%Note, this measure of an iteration is
%generous in that often it counts numerous sequential column additions
%as a single iteration, while at at other times this measure is
%pessimistic as it counts as an iteration checking that there are no
%prior columns which can be added, e.g. determining that a negative
%column to already reduced.  We measure as an iteration of Alg.\
%\ref{alg:alpha_beta} one pass through Phase I and II.  This notion of
%an iteration is consistent with the embarrassingly parallel nature of
%Phase II where a numerous pairs of columns can be added
%simultaneously; however the notion is also generous as it does not
%consider possible limitations to the number of processors available
%and includes the 
%sequential Phase I within the same iteration, which we do so as Phase
%1 does not include any column additions.  While the above notions are 
%inconsistent, they are chosen to avoid details of the Phase I search
%cost as compared to that of a column addition, and as a compromise to
%considering an instance of, possibly parallel, column additions as an
%iteration.  The numerical results presented in this section simulate
%a parallel implementation and are available at \cite{githubRepo}. 
%A truly parallel implementation is being developed.

\subsection{Operations are packed in fewer iterations}

For each ensemble we sample three point-clouds and reduce their
corresponding boundary matrices using each of Algs.\ \ref{alg:mat_red}
- \ref{alg:twist}.
Fig.\ \ref{fig:operation_count_benchmark} illustrates the
computational complexity, as the number of column 
additions performed at each iteration and also the number of
nontrivial {\em xor} operations when adding columns in the matrix;
that is, the number of scalar operations of the form 
\[
\left\{1 \oplus 1, 1 \oplus 0, 0 \oplus 1\right\}.
\]
Scalar operations are relevant because in some storage models like the
one given in \cite{chen2011persistent} the computational cost of
updating the matrix depends super-linearly on the number of non-zeros
in the columns being added. 
%
%Our results are shown in Figure \ref{fig:operation_count_benchmark}.
%
Figs.\ \ref{fig:col_adds_gaussian}, \ref{fig:col_adds_trefoil} and
\ref{fig:col_adds_sphere} show the number of column additions per
iteration of each algorithm performs to reduce each of the simplicial
compelexes under consideration. 
Specifically, they illustrate how that Alg.\ \ref{alg:alpha_beta}
allocates most of the column additions at the earliest iterations due
to its highly parallel nature.
On the other hand, Figs.\ \ref{fig:col_adds_cumulative_gaussian},
\ref{fig:col_adds_cumulative_trefoil}, and
\ref{fig:col_adds_cumulative_sphere}  as well as Tab.\
\ref{tab:ratio_cumsum_operations} 
show that while Alg.\
\ref{alg:alpha_beta} reduces the number of iterations, it has a total
number of column additions indistingushable to Alg.\ \ref{alg:twist}
and about half that of Alg.\ \ref{alg:mat_red}.
This shows that Alg.\ \ref{alg:alpha_beta} is successful in
packing the number of left-to-right column operations into
considerably few independent iterations. 

\begin{figure*}[!htbp]%[H]
	\centering
%\subfloat[Icosahedron]{\includegraphics[width=0.33\linewidth]{{benchmark_alpha_beta_parallel-icosahedron-num_col_adds}}\label{fig:col_adds_icosahedron}}
%\subfloat[Morozov order 7]{\includegraphics[width=0.33\linewidth]{{benchmark_alpha_beta_parallel-morozov-num_col_adds}}\label{fig:col_adds_morozov}}
%\subfloat[Icosahedron]{\includegraphics[width=0.33\linewidth]{{benchmark_alpha_beta_parallel-icosahedron-num_col_adds_cumulative}}\label{fig:col_adds_cumulative_icosahedron}}
%\subfloat[Morozov order 7]{\includegraphics[width=0.33\linewidth]{{benchmark_alpha_beta_parallel-morozov-num_col_adds_cumulative}}\label{fig:col_adds_cumulative_morozov}}
%\subfloat[Icosahedron]{\includegraphics[width=0.33\linewidth]{{benchmark_alpha_beta_parallel-icosahedron-num_entry_adds}}\label{fig:num_entry_adds_icosahedron}}
%\subfloat[Morozov order 7]{\includegraphics[width=0.33\linewidth]{{benchmark_alpha_beta_parallel-morozov-num_entry_adds}}\label{fig:num_entry_adds_morozov}}
%\subfloat[Icosahedron]{\includegraphics[width=0.33\linewidth]{{benchmark_alpha_beta_parallel-icosahedron-num_entry_adds_cumulative}}\label{fig:num_entry_adds_cumulative_icosahedron}}
%\subfloat[Morozov order 7]{\includegraphics[width=0.33\linewidth]{{benchmark_alpha_beta_parallel-morozov-num_entry_adds_cumulative}}\label{fig:num_entry_adds_cumulative_morozov}}
\subfloat[Gaussian; Count of column additions][Gaussian \\ Count of column additions]{\includegraphics[width=0.33\linewidth]{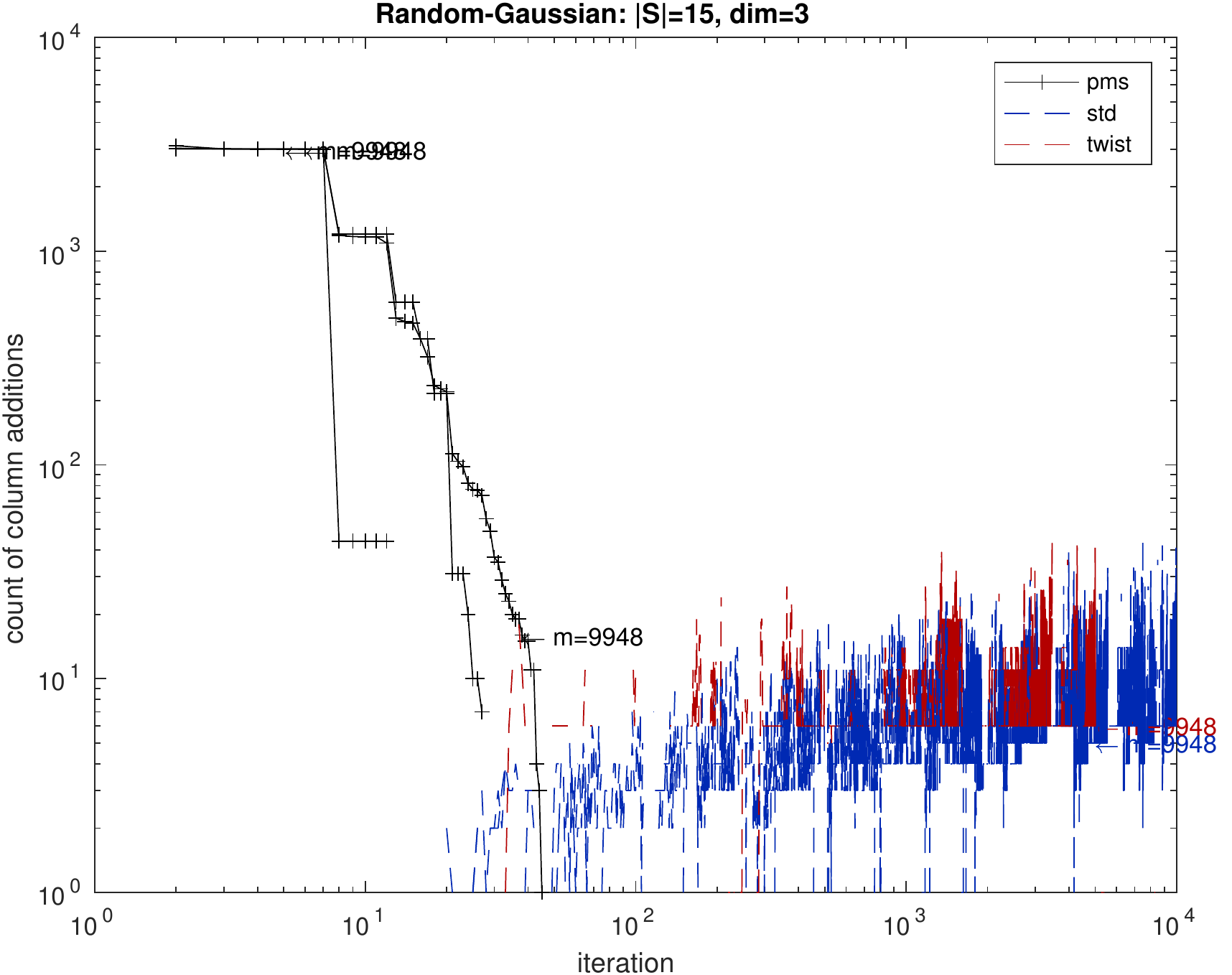}\label{fig:col_adds_gaussian}}
\subfloat[Gaussian; Count of column additions (cumulative)][Gaussian \\ Count of column additions (cumulative)]{\includegraphics[width=0.33\linewidth]{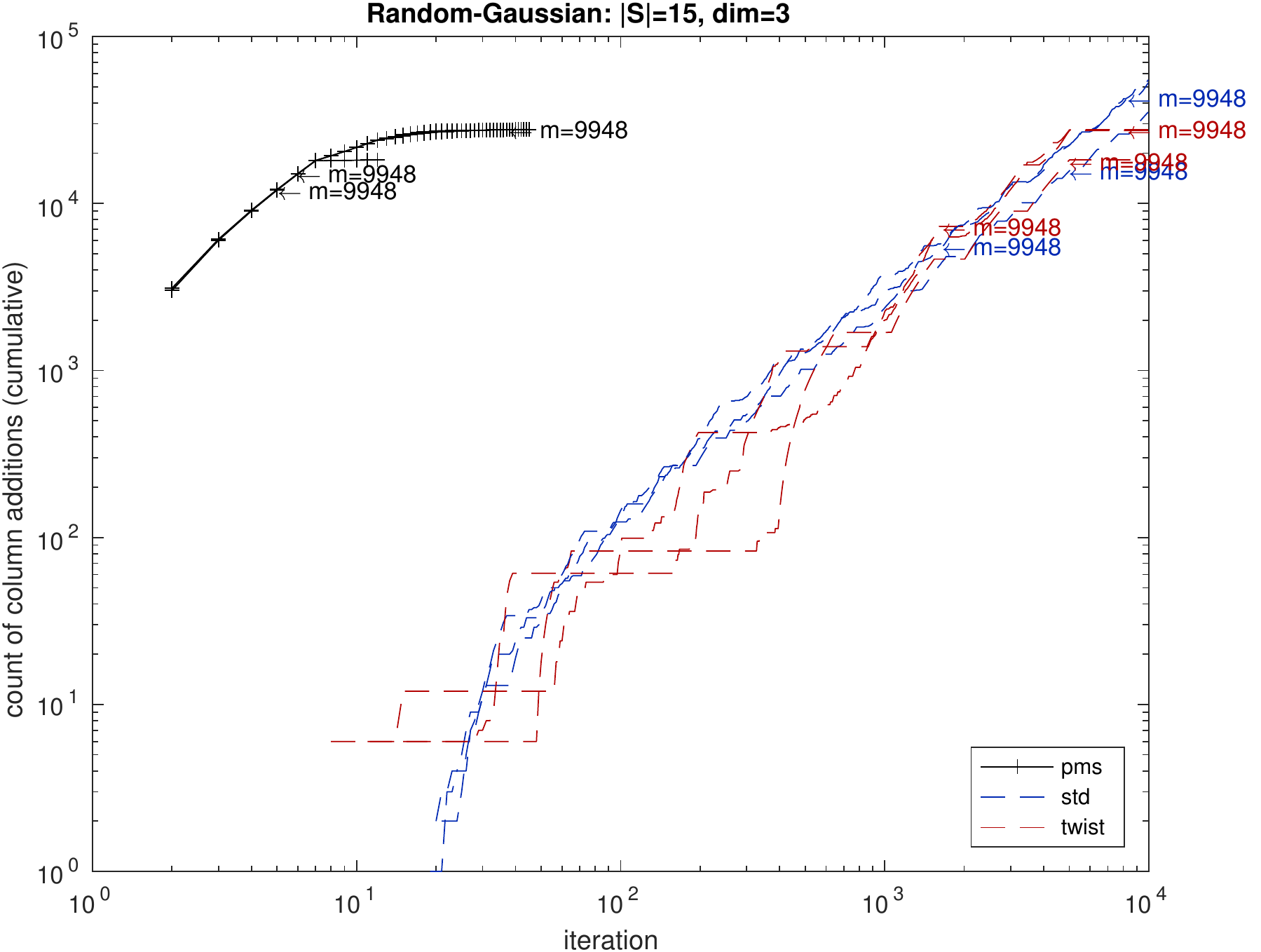}\label{fig:col_adds_cumulative_gaussian}}
%\subfloat[Gaussian]{\includegraphics[width=0.33\linewidth]{{benchmark_alpha_beta_parallel-random_gaussian-num_entry_adds}}\label{fig:num_entry_adds_gaussian}}
\subfloat[Gaussian; Count of {\tt XOR} operations (cumulative)][Gaussian \\ Count of {\tt XOR} operations (cumulative)]{\includegraphics[width=0.33\linewidth]{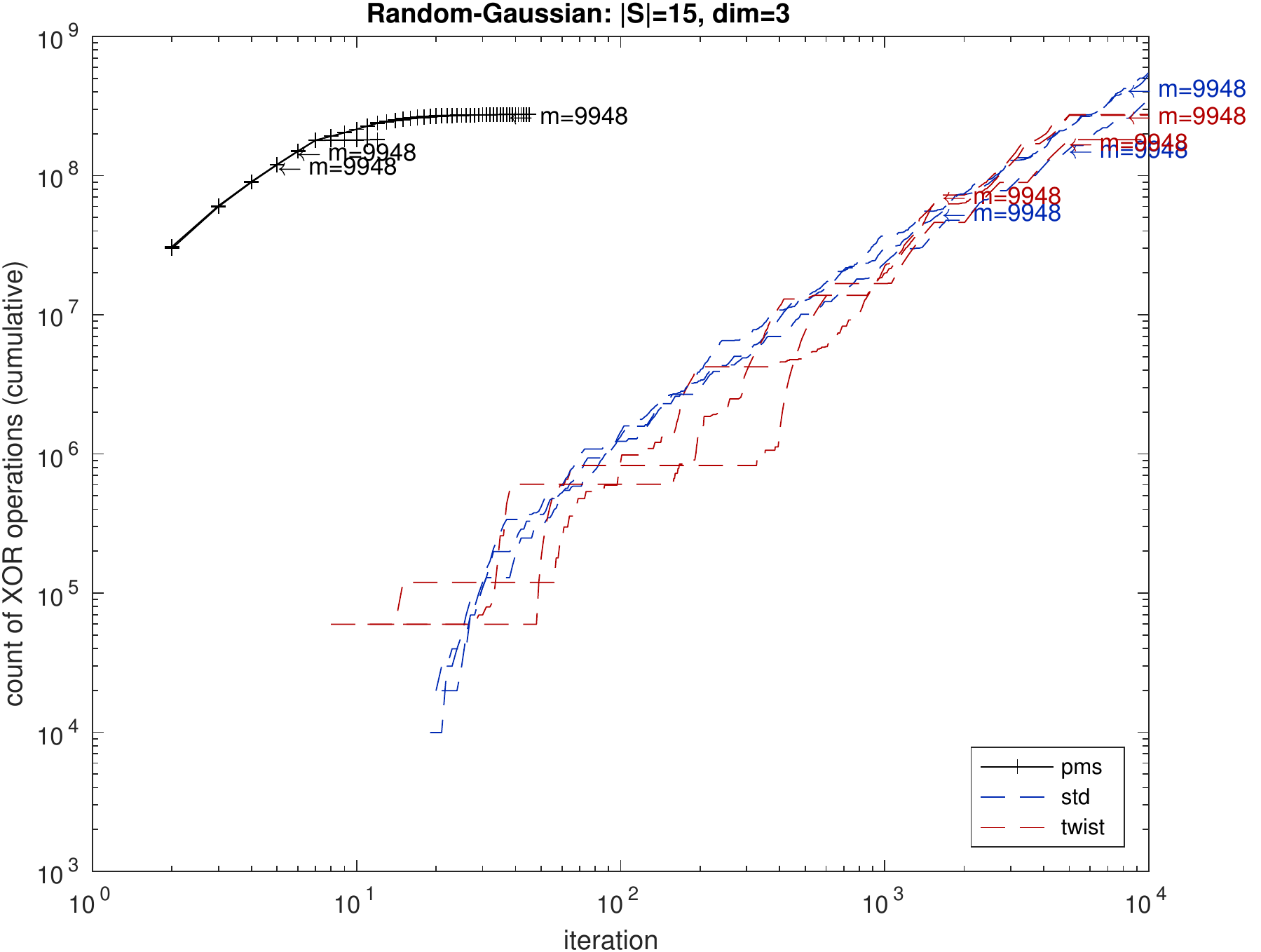}\label{fig:num_entry_adds_cumulative_gaussian}}\\

\subfloat[Figure-8; Count of column additions][Figure-8 \\ Count of column additions]{\includegraphics[width=0.33\linewidth]{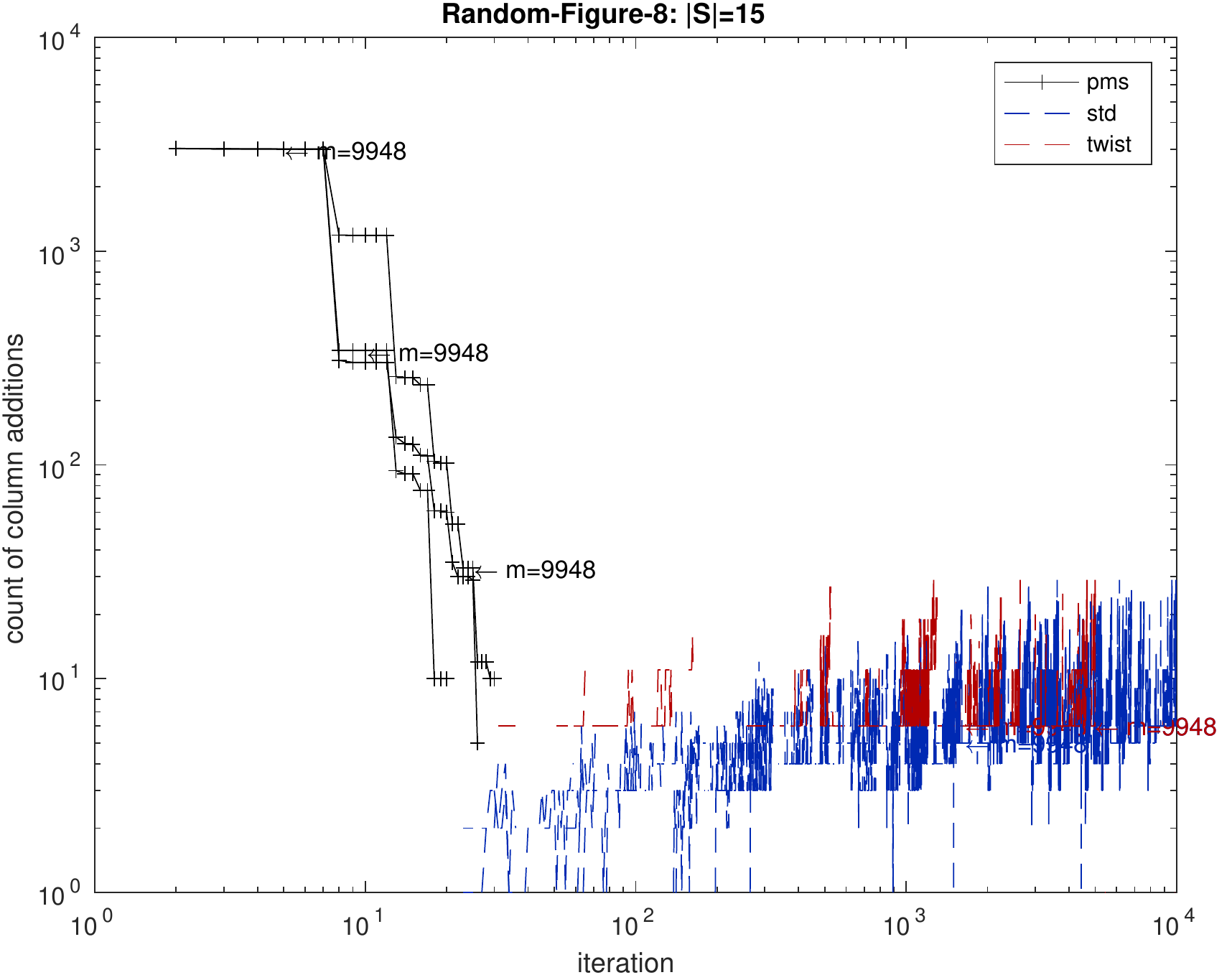}\label{fig:col_adds_f8}}
\subfloat[Figure-8; Count of column additions (cumulative)][Figure-8 \\ Count of column additions (cumulative)]{\includegraphics[width=0.33\linewidth]{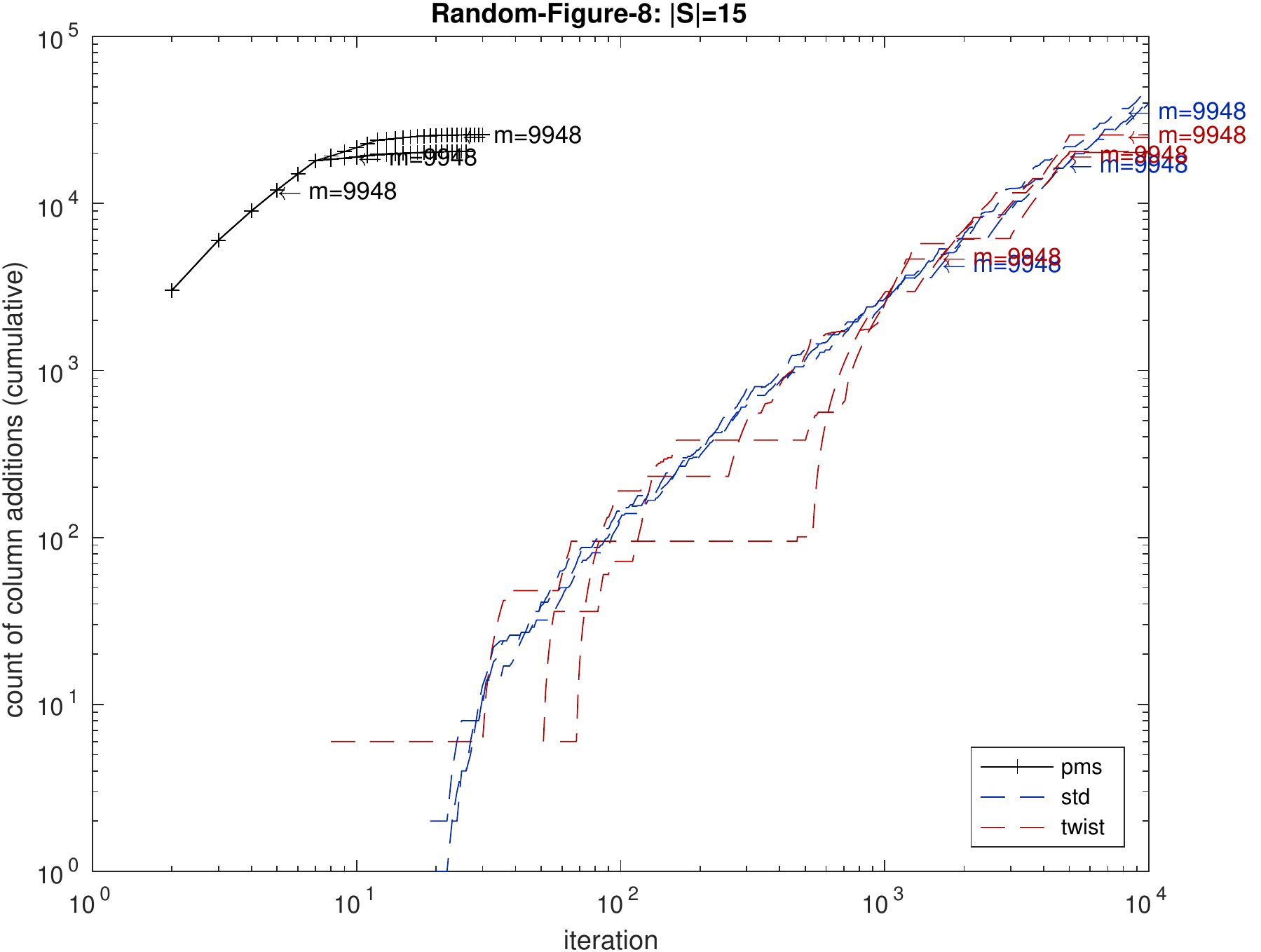}\label{fig:col_adds_cumulative_f8}}
%\subfloat[Figure-8]{\includegraphics[width=0.33\linewidth]{{benchmark_alpha_beta_parallel-random_figure_8-num_entry_adds}}\label{fig:num_entry_adds_f8}}
\subfloat[Figure-8; Count of {\tt XOR} operations (cumulative)][Figure-8 \\ Count of {\tt XOR} operations (cumulative)]{\includegraphics[width=0.33\linewidth]{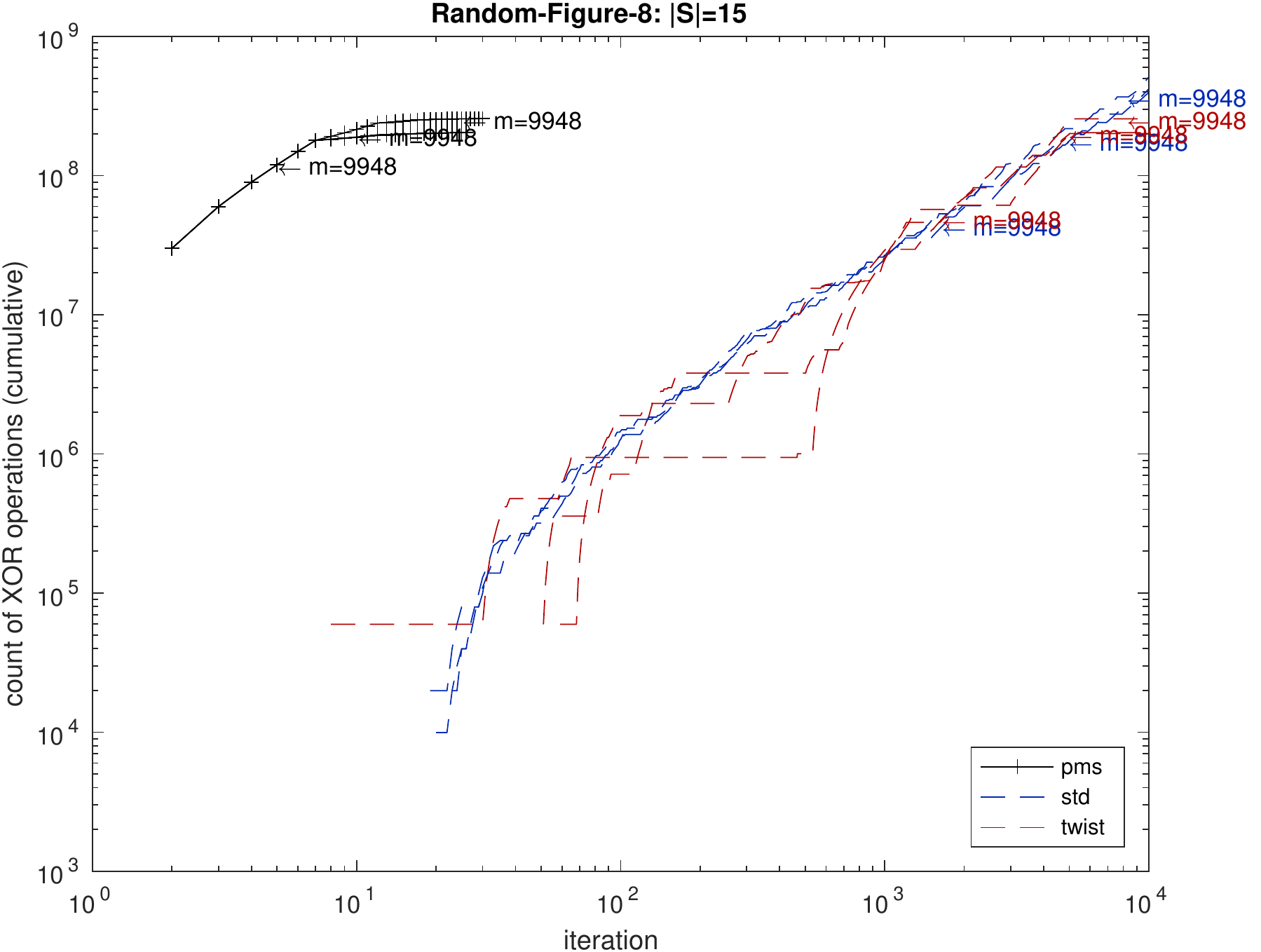}\label{fig:num_entry_adds_cumulative_f8}}\\

\subfloat[Trefoil Knot; Count of column additions][Trefoil Knot \\ Count of column additions]{\includegraphics[width=0.33\linewidth]{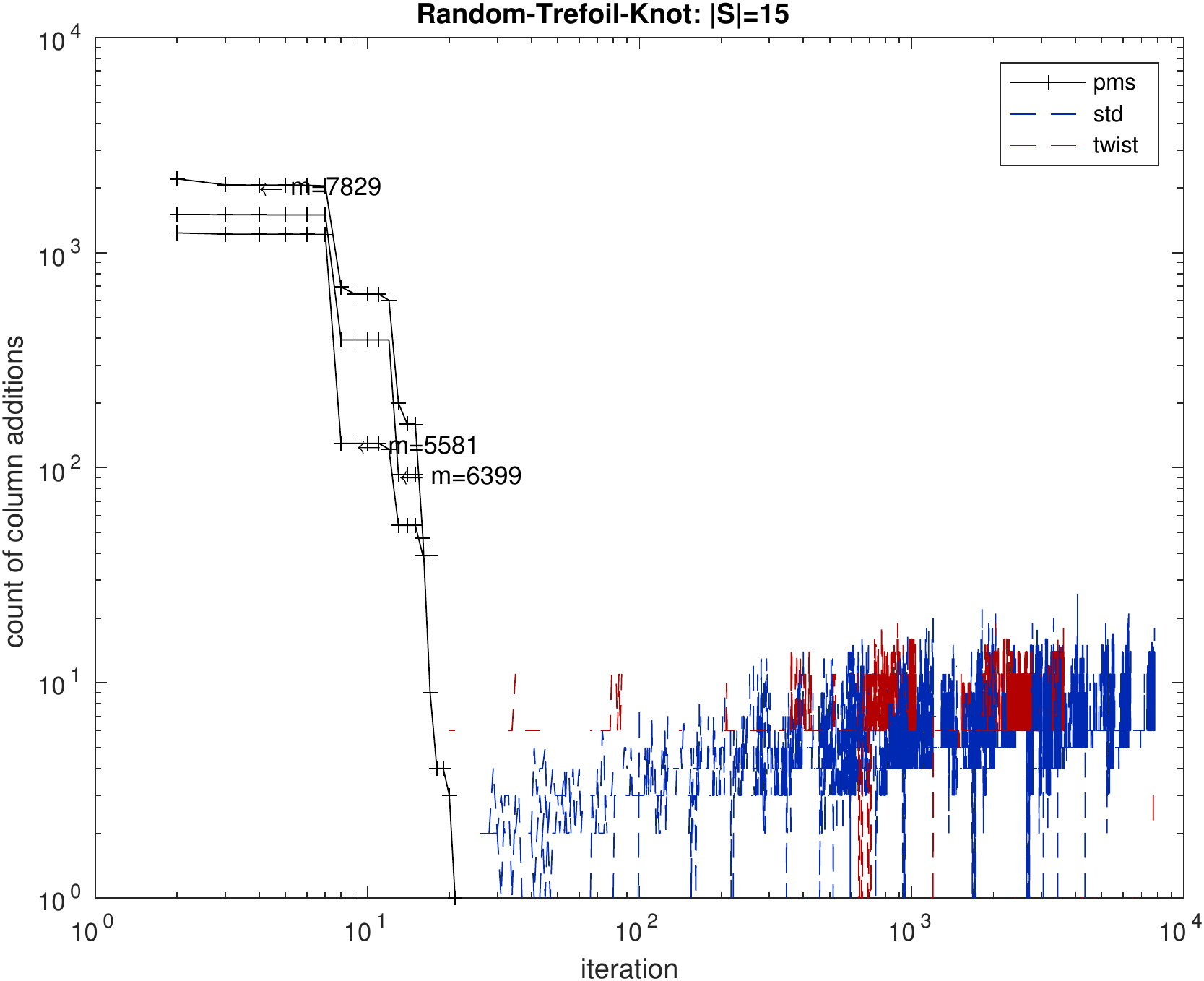}\label{fig:col_adds_trefoil}}
\subfloat[Trefoil Knot; Count of column additions (cumulative)][Trefoil Knot \\ Count of column additions (cumulative)]{\includegraphics[width=0.33\linewidth]{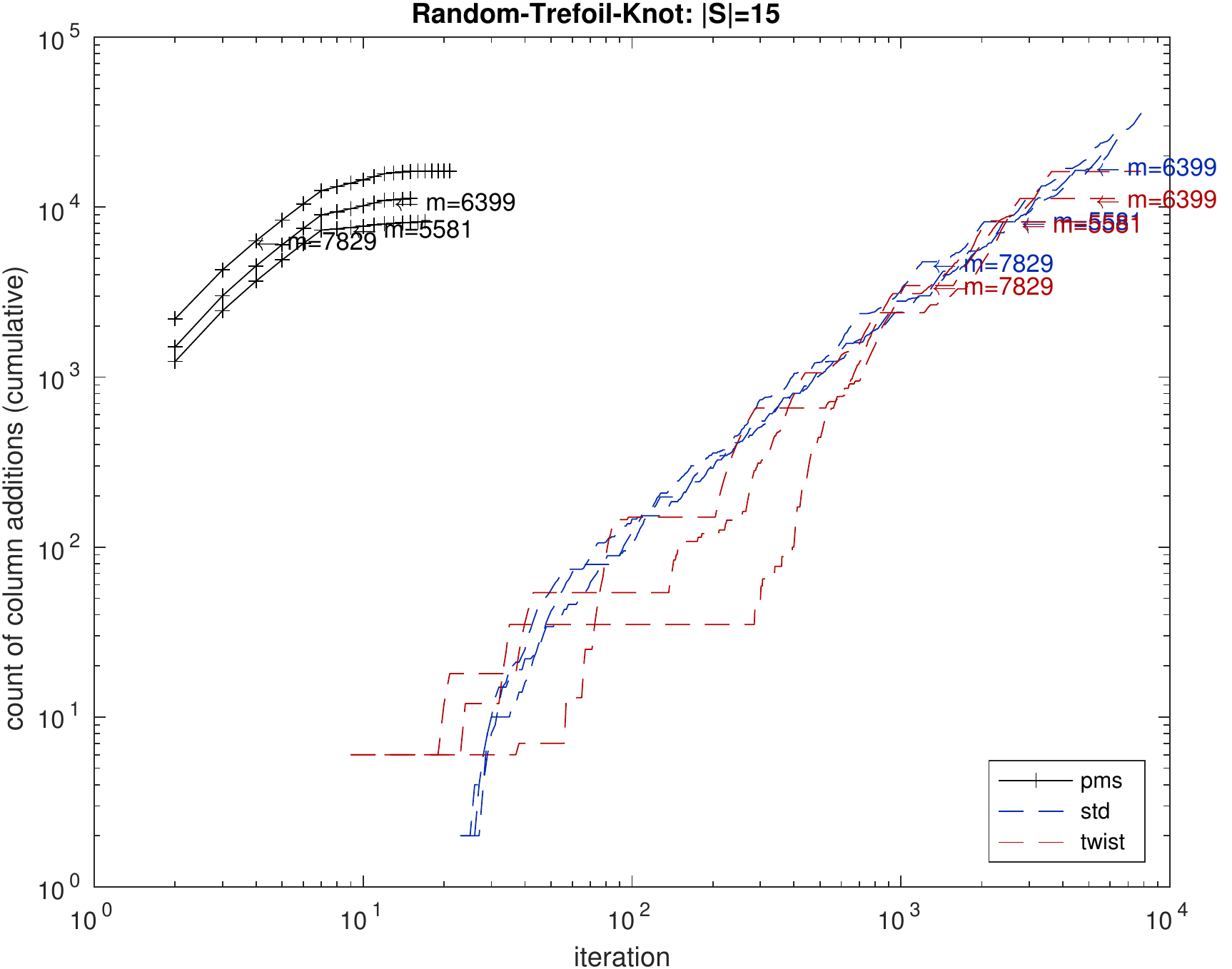}\label{fig:col_adds_cumulative_trefoil}}
%\subfloat[Trefoil Knot]{\includegraphics[width=0.33\linewidth]{{benchmark_alpha_beta_parallel-random_trefoil_knot-num_entry_adds}}\label{fig:num_entry_adds_trefoil}}
\subfloat[Trefoil Knot; Count of {\tt XOR} operations (cumulative)][Trefoil Knot \\ Count of {\tt XOR} operations (cumulative)]{\includegraphics[width=0.33\linewidth]{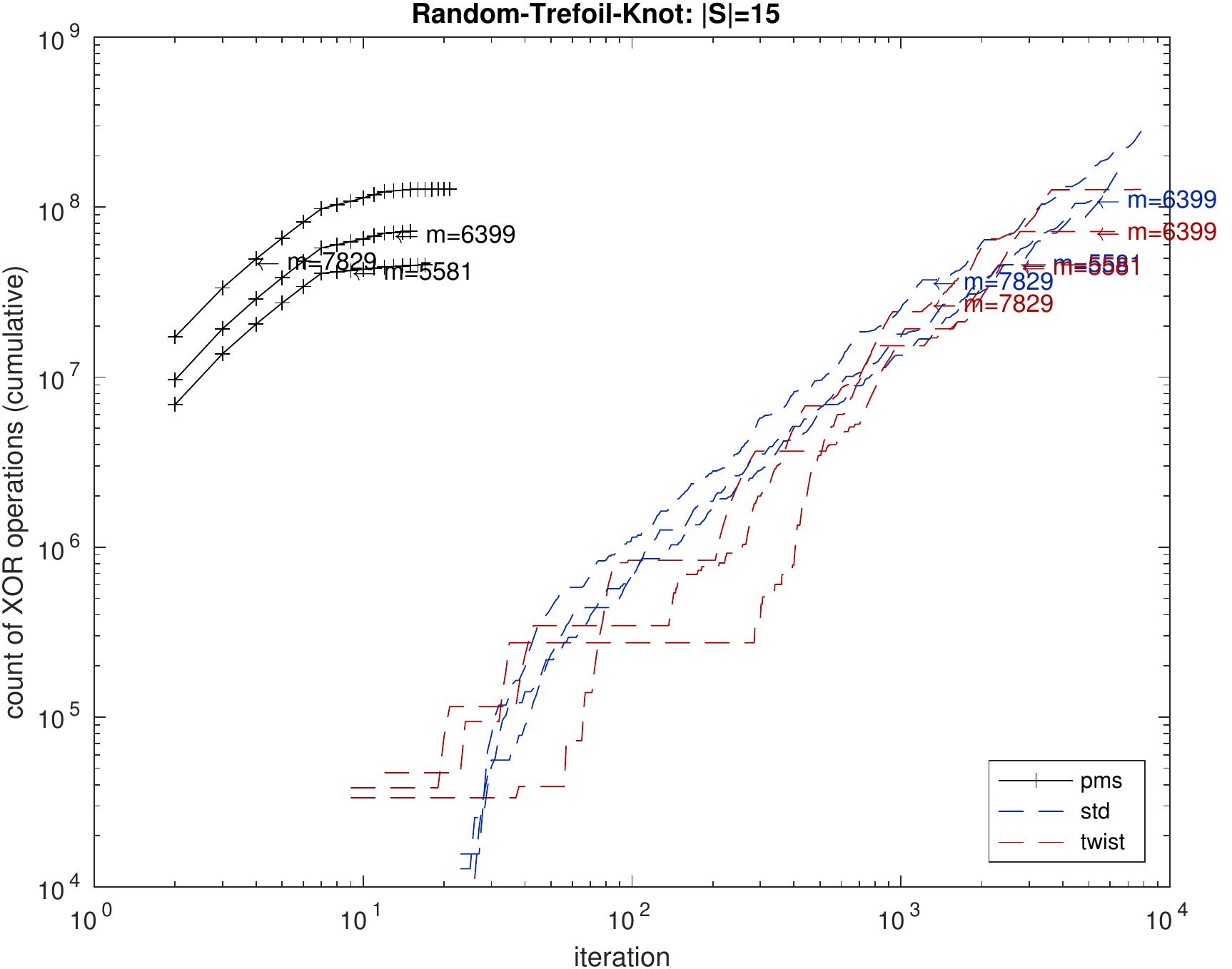}\label{fig:num_entry_adds_cumulative_trefoil}}\\

\subfloat[Sphere product; Count of column additions][Sphere product \\ Count of column additions]{\includegraphics[width=0.33\linewidth]{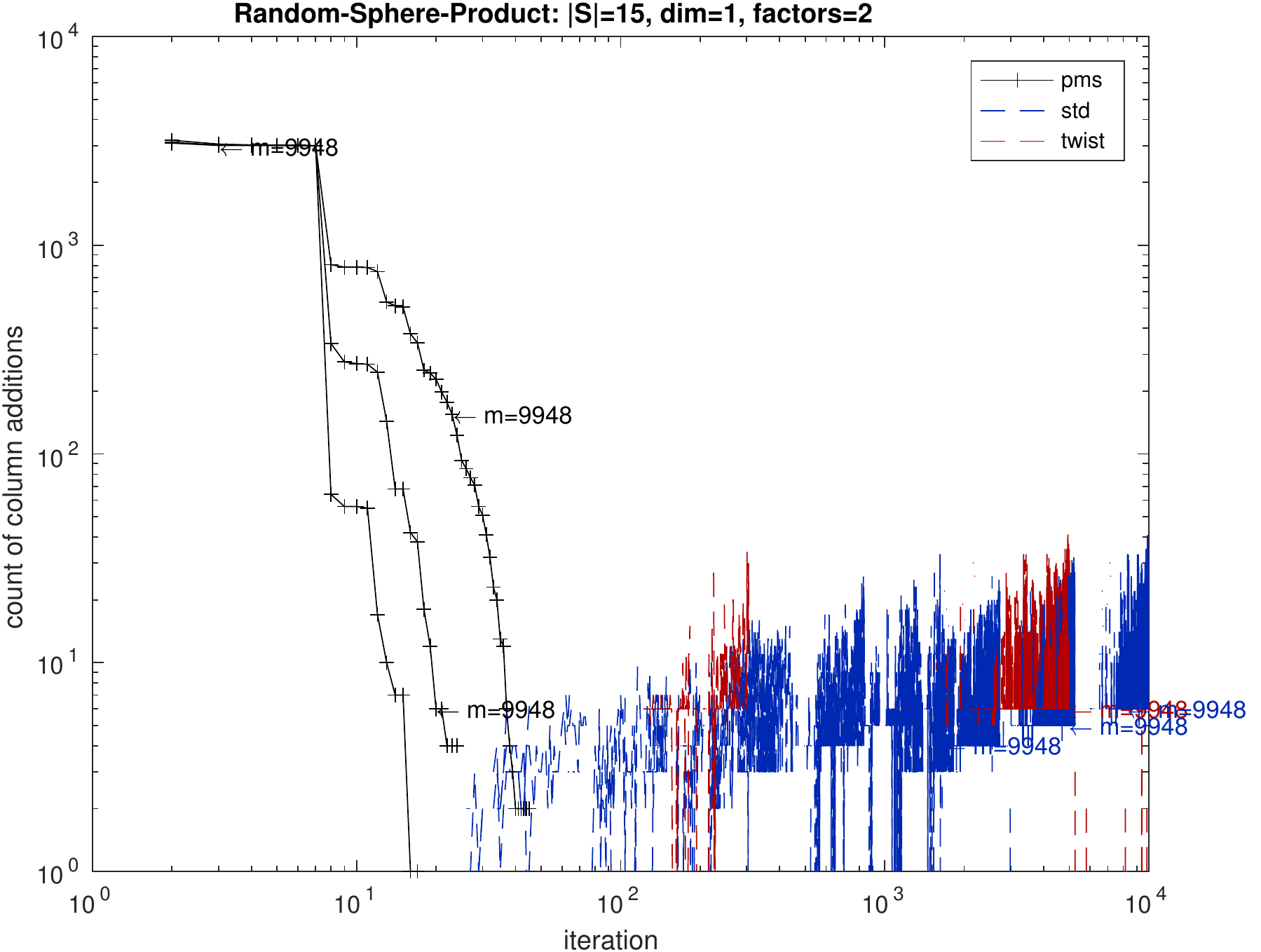}\label{fig:col_adds_sphere}}
\subfloat[Sphere product; Count of column additions (cumulative)][Sphere product \\ Count of column additions (cumulative)]{\includegraphics[width=0.33\linewidth]{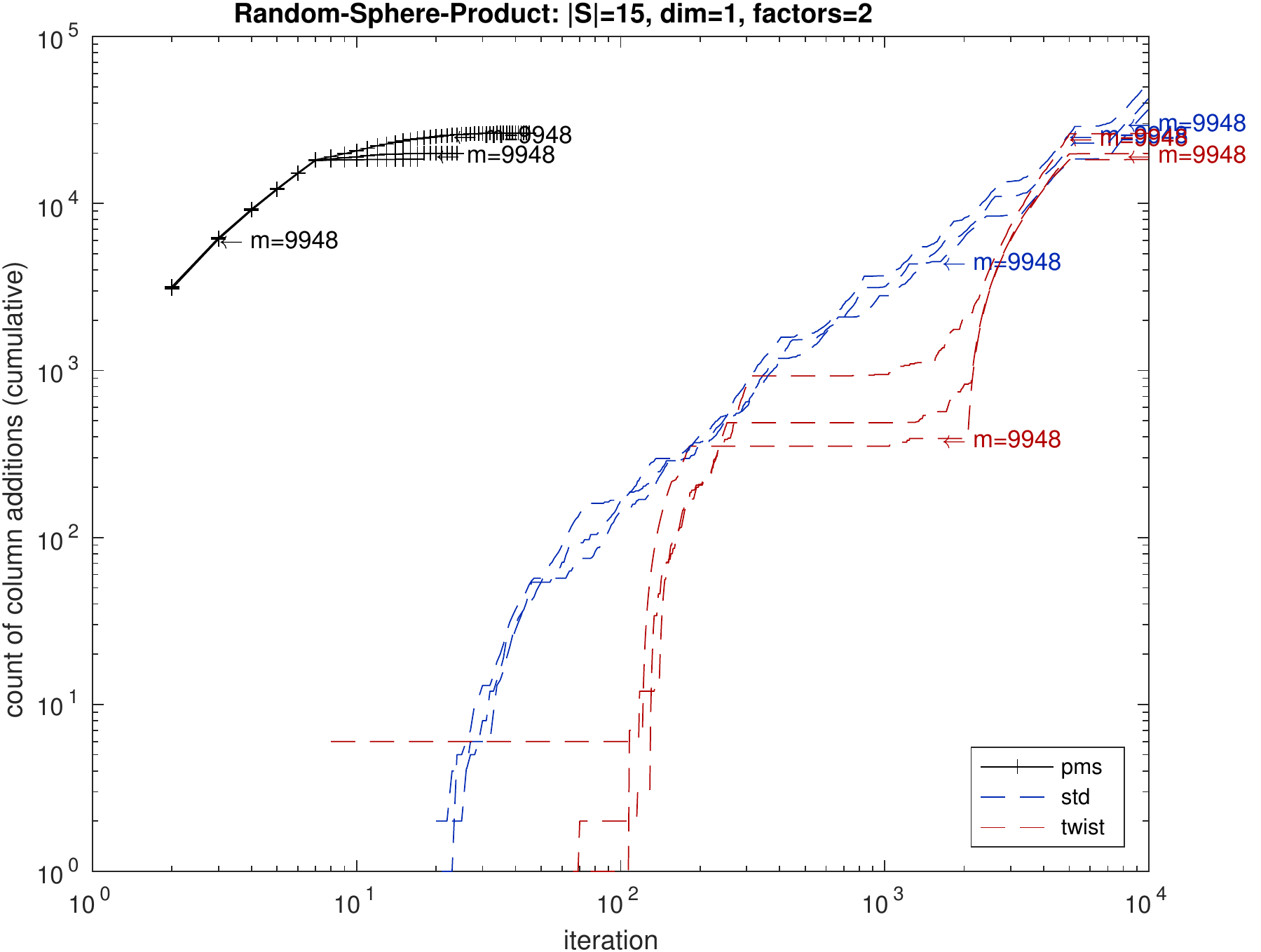}\label{fig:col_adds_cumulative_sphere}}
%\subfloat[Sphere product]{\includegraphics[width=0.33\linewidth]{{benchmark_alpha_beta_parallel-sphere_product-num_entry_adds}}\label{fig:num_entry_adds_sphere}}
\subfloat[Sphere product; Count of {\tt XOR} operations (cumulative)][Sphere product \\ Count of {\tt XOR} operations (cumulative)]{\includegraphics[width=0.33\linewidth]{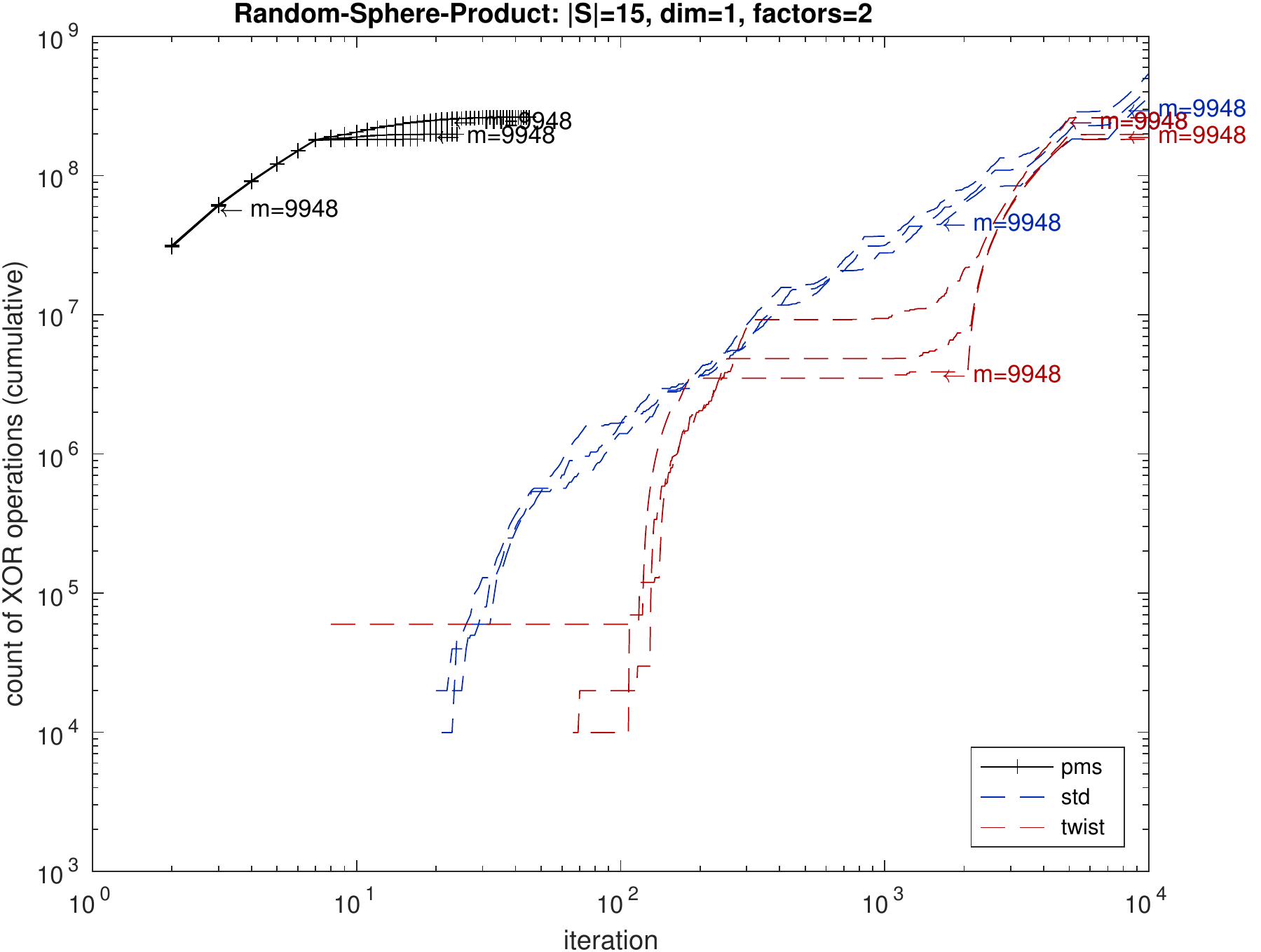}\label{fig:num_entry_adds_cumulative_sphere}}
\caption[Benchmarking on column operation overhead]{{\bf Benchmarking
    on column operation overhead.} For each ensemble, three point
  clouds are sampled and their corresponding simplicial complexes are
  reduced with Algs.\ \ref{alg:mat_red} - \ref{alg:twist}. Their
  performance is benchmarked in three different ways: through the
  number of column additions done {\em at each} iteration (Figs.\
  \ref{fig:col_adds_gaussian}, \ref{fig:col_adds_f8},
  \ref{fig:col_adds_trefoil}, \ref{fig:col_adds_sphere}); through
  the cumulative number of column additions {\em up to} a given
  iteration (Figs.\  \ref{fig:col_adds_cumulative_gaussian},
  \ref{fig:col_adds_cumulative_f8},
  \ref{fig:col_adds_cumulative_trefoil},
  \ref{fig:col_adds_cumulative_sphere}); and through the cumulative number of {\tt XOR} operations done {\em up to } a given iteration (Figs.\  \ref{fig:num_entry_adds_cumulative_gaussian}, \ref{fig:num_entry_adds_cumulative_f8}, \ref{fig:num_entry_adds_cumulative_trefoil}, \ref{fig:num_entry_adds_cumulative_sphere}).}
\label{fig:operation_count_benchmark}
\end{figure*}

\begin{small}
\begin{table*}
\centering
\begin{tabular}{l||cc|cc|cc|cc}
\toprule
\multirow{2}{*}{Sample} &
\multicolumn{2}{c}{Gaussian}&
\multicolumn{2}{c}{Figure-8}&
\multicolumn{2}{c}{Trefoil-knot}&
\multicolumn{2}{c}{Sphere-product}\\
& {std} & {twist} & {std} & {twist} & {std} & {twist} & {std} & {twist} \\
\midrule
1 & 0.52 & 1.00 & 0.49 & 1.00 & 0.46 & 1.00 & 0.50 & 1.00\\
2 & 0.52 & 1.00 & 0.51 & 1.00 & 0.40 & 1.00 & 0.49 & 1.01\\
3 & 0.50 & 1.00 & 0.50 & 1.00 & 0.44 & 1.00 & 0.47 & 1.00\\
\bottomrule
\end{tabular}
\caption{{\bf Ratio of total column additions.} For each ensemble,
  three point clouds are sampled and the corresponding simplicial
  complexes are reduced with each algorithm. The {\em ratio of total
    column additions} between Alg.\ \ref{alg:alpha_beta} and both
  Algs.\ \ref{alg:mat_red} and \ref{alg:twist} is reported.}
\label{tab:ratio_cumsum_operations}
\end{table*}
\end{small}

\subsection{$\lowstar$ is approximated at multiple scales}
\label{sec:lowstar_convergence}

Let $\low^{\ell}$ be the estimate of $\lowstar$ at the $\ell$-th
iteration of an algorithm.
We evaluate the quality of the approximation by computing the relative $\ell_1$-error as
\begin{equation}\label{eq:rel_error}
\error^{\ell} = \frac{\|\low^{\ell} - \lowstar\|_1}{\|\lowstar\|_1}.
\end{equation}
Fig.\ \ref{fig:l1_distance} illustrates an improved rate of reduction
in $\error^{k}$ per iteration as each of the algorithms progresses.
This is achieved without resorting to further prioritisation of this
reduction as described in Sec.\ \ref{sec:termination} as we are
simulating a number of processors in excess of the number of column
additions needed per iteration.  Tab.\ \ref{tab:iterations_l1_error}
shows the precise number of iterations each of Algs.\
\ref{alg:mat_red} - \ref{alg:twist} need to achieve
\eqref{eq:rel_error} less than $10^{-k}$ for $k=1,2,3,4$ and complete
reduction.  Tab.\ \ref{tab:iterations_l1_error} shows the rapid
reduction of \eqref{eq:rel_error} by Alg.\ \ref{alg:alpha_beta}
compared to only appreciable reduction by Algs.\ \ref{alg:mat_red} and
\ref{alg:twist} near their complete reductions; e.g. wheras a relative
reduction of \eqref{eq:rel_error} to $1\%$ is achieved by Alg.\
\ref{alg:alpha_beta}
in less than 20 iterations, Algs.\ \ref{alg:mat_red} and \ref{alg:twist}
take approximately ten and nine thousand iterations respectively for
the same reduction.
Fig. \ref{fig:percentage_unreduced} and Tab.\
\ref{tab:iterations_unreduced} show a similarly early decrease in the 
fraction of the number of columns which are fully reduced.
Remarkably, Alg.\ \ref{alg:alpha_beta} has reduced at least $50\%$ of
the columns in only two iterations, and all but $1\%$ within no more
than 20 iterations; this is in contrast with between approximately two
and a half and ten thousand iterations  for compreable reductions by
Algs.\ \ref{alg:mat_red} and \ref{alg:twist}.
%
%Evidently, Algorithm \ref{alg:alpha_beta} outperforms other algorithms due to its parallel nature.
%

\begin{figure*}[!htbp]%[H]
	\centering
%\subfloat[Icosahedron]{\includegraphics[width=0.50\linewidth]{{benchmark_alpha_beta_parallel-icosahedron-lowstar_l1_distance}}\label{fig:l1_icosahedron}}
%\subfloat[Morozov order 7]{\includegraphics[width=0.50\linewidth]{{benchmark_alpha_beta_parallel-morozov-lowstar_l1_distance}}\label{fig:l1_morozov}}
\subfloat[Gaussian]{\includegraphics[width=0.50\linewidth]{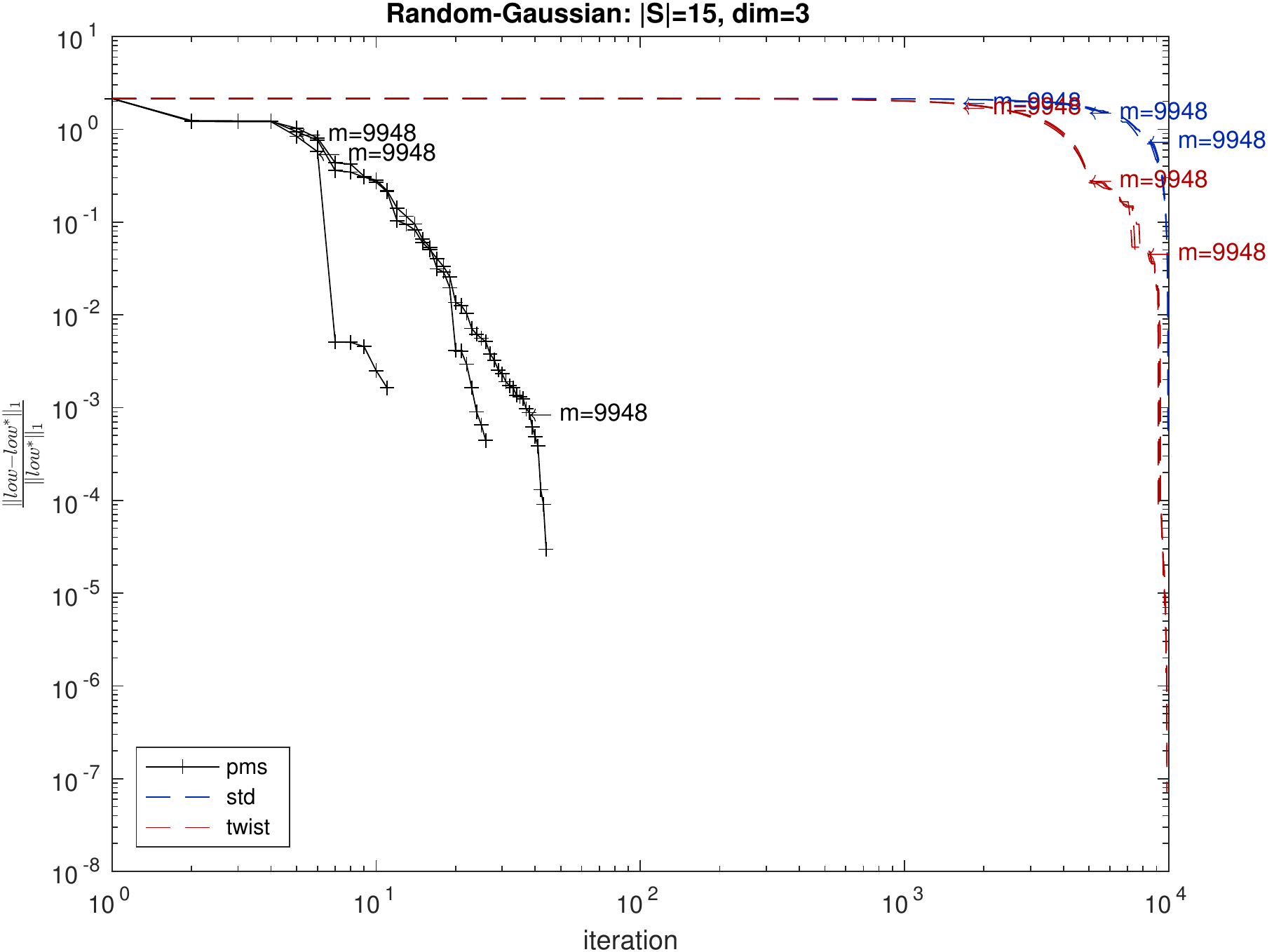}\label{fig:l1_gaussian}}
\subfloat[Figure-8]{\includegraphics[width=0.50\linewidth]{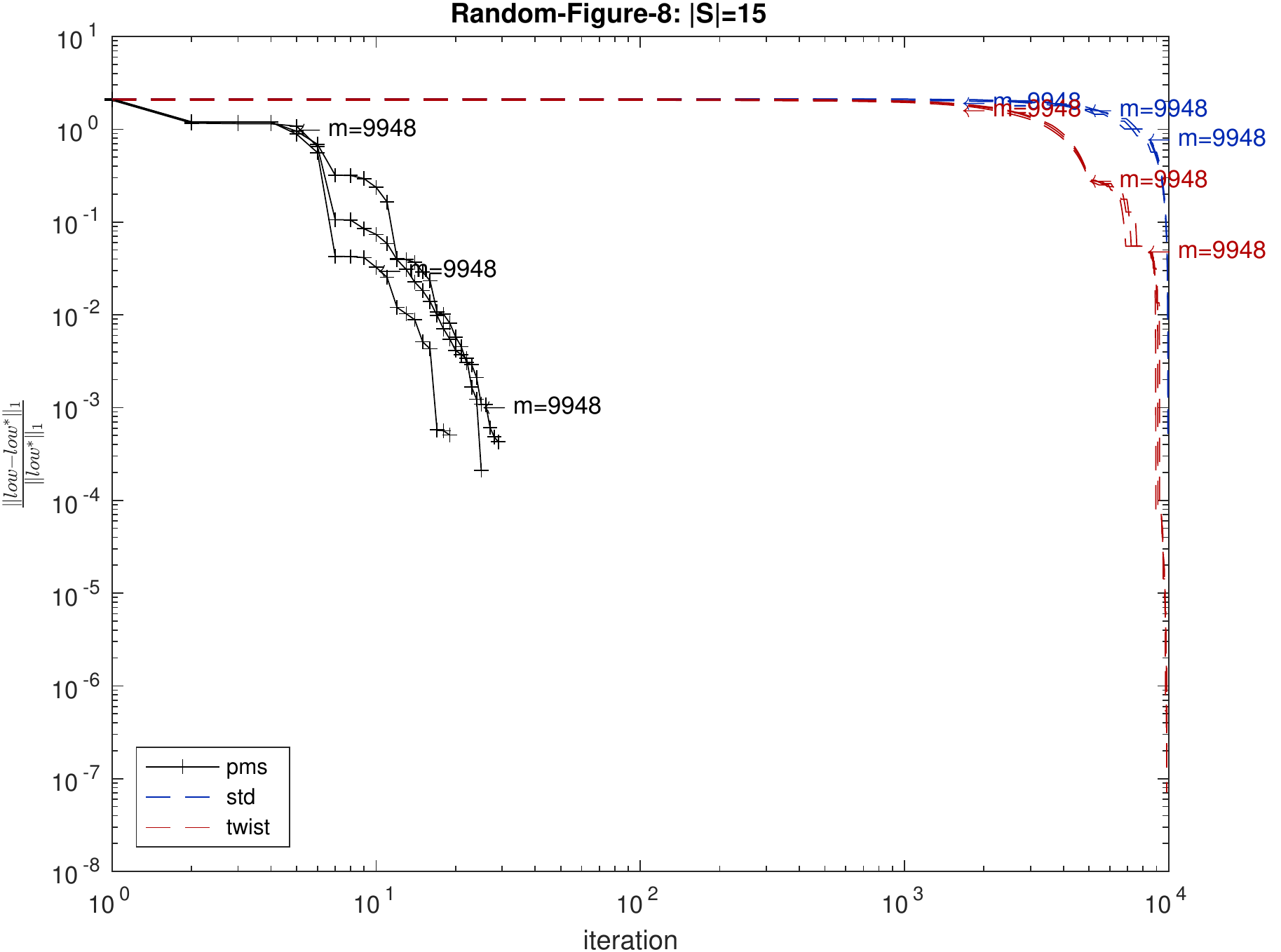}\label{fig:l1_f8}}\\
\subfloat[Trefoil Knot]{\includegraphics[width=0.50\linewidth]{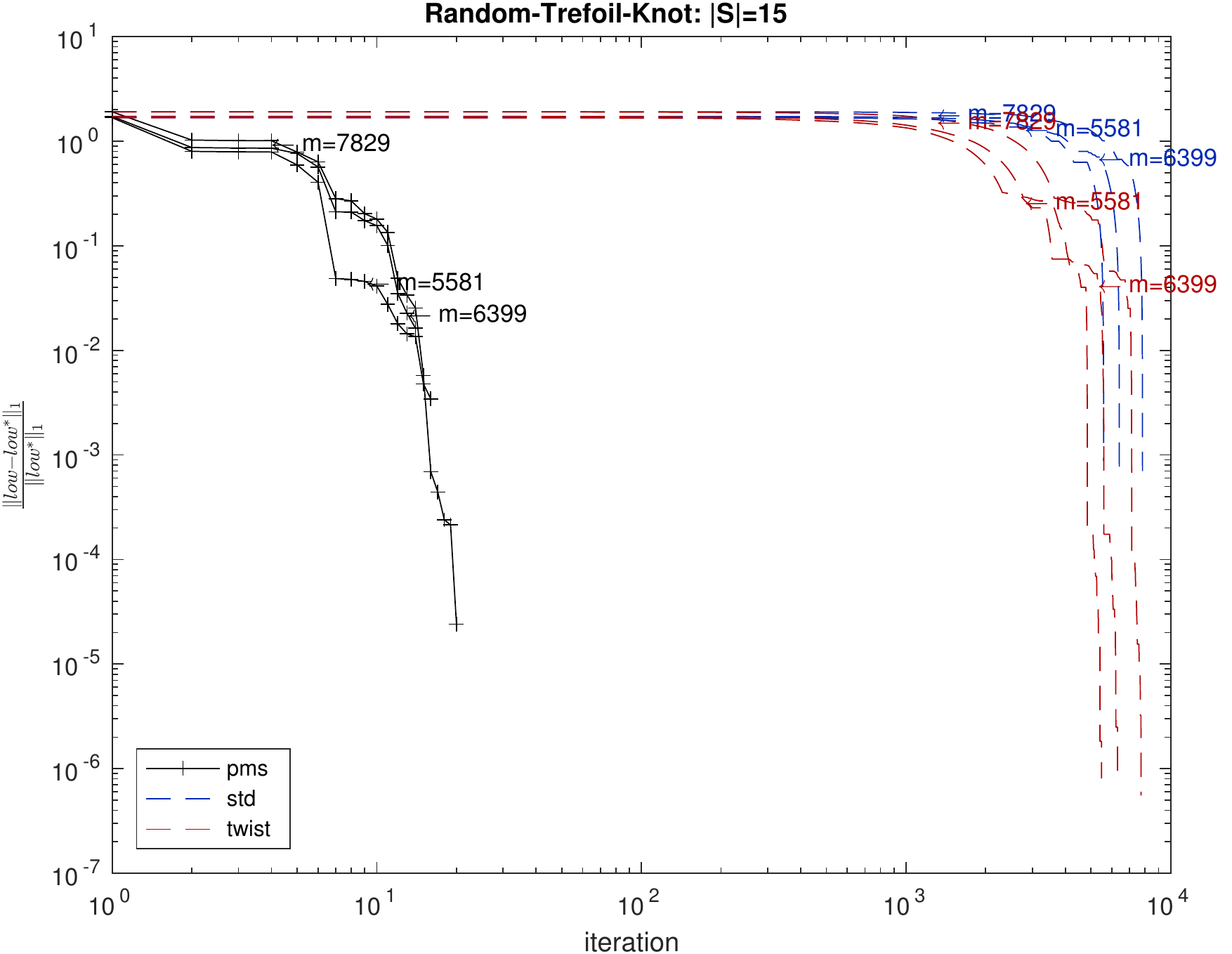}\label{fig:l1_trefoil}}
\subfloat[Sphere product]{\includegraphics[width=0.50\linewidth]{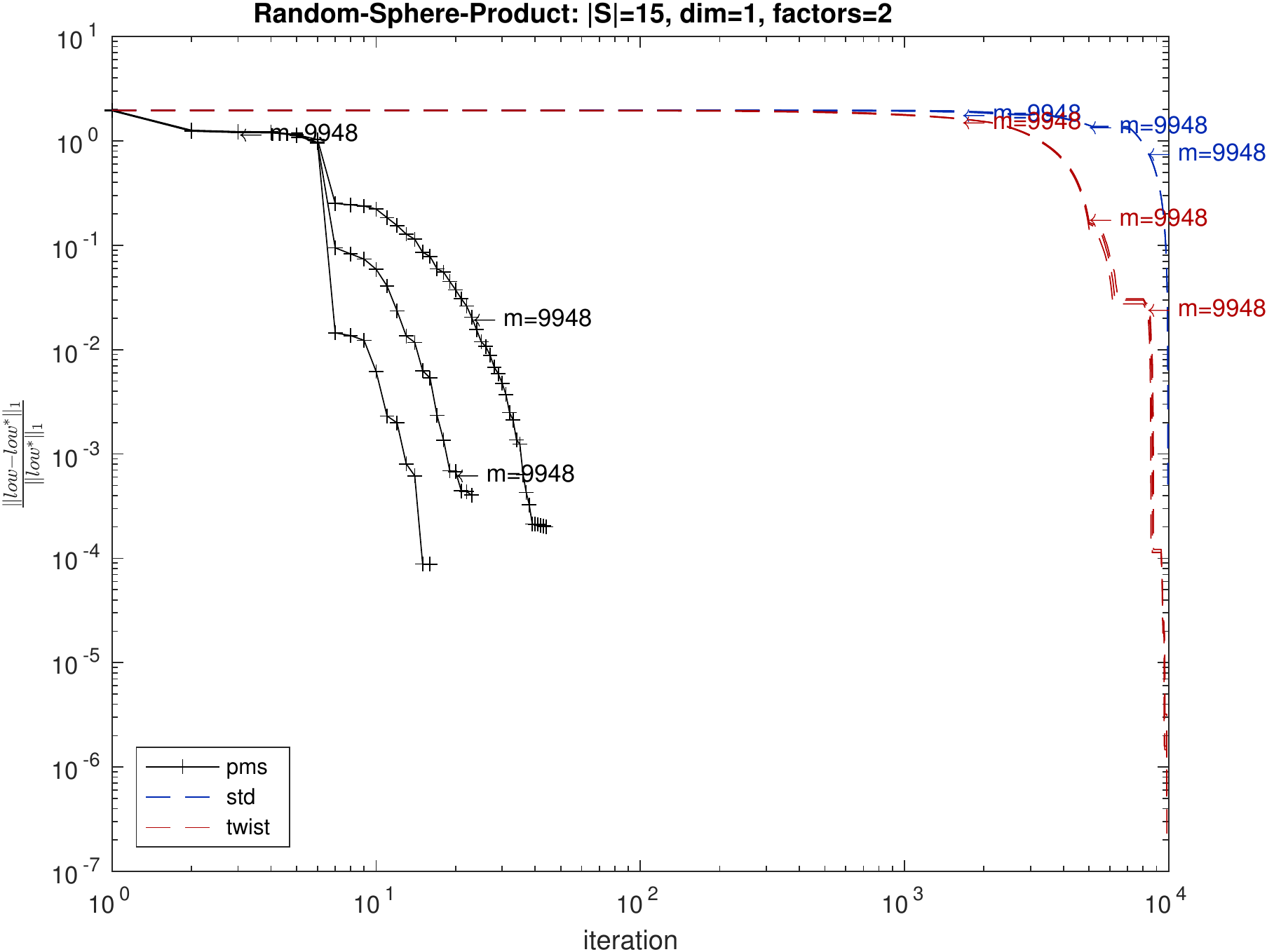}\label{fig:l1_sphere}}
\caption[Relative $\ell_1$-error per iteration]{{\bf Relative
    $\ell_1$-error per iteration.} For each ensemble, three point
  clouds are sampled and the corresponding simplicial complexes are
  reduced with each of the algorithms. The relative $\ell_1$-error
  between $\low$ and $\lowstar$, as given in \eqref{eq:rel_error}, is tracked.}
	\label{fig:l1_distance}
\end{figure*}

\begin{small}
\begin{table*}
\centering
\begin{tabular}{l||ccc|ccc|ccc|ccc}
\toprule
\multirow{2}{*}{$\frac{\|\low - \lowstar\|_1}{\|\lowstar\|_1}$} &
\multicolumn{3}{c}{Gaussian}&
\multicolumn{3}{c}{Figure-8}&
\multicolumn{3}{c}{Trefoil-knot}&
\multicolumn{3}{c}{Sphere-product}\\
& {pms} & {std} & {twist} & {pms} & {std} & {twist} & {pms} & {std} & {twist} & {pms} & {std} & {twist} \\
\midrule
0.1 & 14 & 9786 & 7366 & 9 & 9756 & 6759 & 12 & 7687 & 5519 & 7 & 9751 & 5589\\
0.01 & 20 & 9933 & 9180 & 17 & 9930 & 8916 & 15 & 7816 & 7087 & 10 & 9930 & 8485\\
0.001 & 24 & 9948 & 9197 & 25 & 9948 & 8942 & 16 & 7829 & 7122 & 13 & 9948 & 8564\\
0.0001 & 27 & 9949 & 9198 & 26 & 9949 & 8944 & 20 & 7830 & 7309 & 15 & 9949 & 9388\\
0 & 27 & 9949 & 9858 & 26 & 9949 & 9869 & 21 & 7830 & 7732 & 17 & 9949 & 9840\\
\bottomrule
\end{tabular}
\caption{{\bf Iterations to relative $\ell_1$-error.} For each
  ensemble, one point cloud is sampled and the corresponding
  simplicial complex is reduced with each algorithm. The number of
  iterations to achieve a given {\em relative $\ell_1$-error level}
  between $\low$ and $\lowstar$, as given by \eqref{eq:rel_error}, is
  reported for of Algs.\ \ref{alg:mat_red} - \ref{alg:twist}.} 
\label{tab:iterations_l1_error}
\end{table*}
\end{small}

\begin{figure*}[!htbp]%[H]
	\centering
%\subfloat[Icosahedron]{\includegraphics[width=0.33\linewidth]{{benchmark_alpha_beta_parallel-icosahedron-percentage_unreduced}}\label{fig:percentage_unreduced_icosahedron}}
%\subfloat[Morozov order 7]{\includegraphics[width=0.33\linewidth]{{benchmark_alpha_beta_parallel-morozov-percentage_unreduced}}\label{fig:percentage_unreduced_morozov}}
\subfloat[Gaussian]{\includegraphics[width=0.50\linewidth]{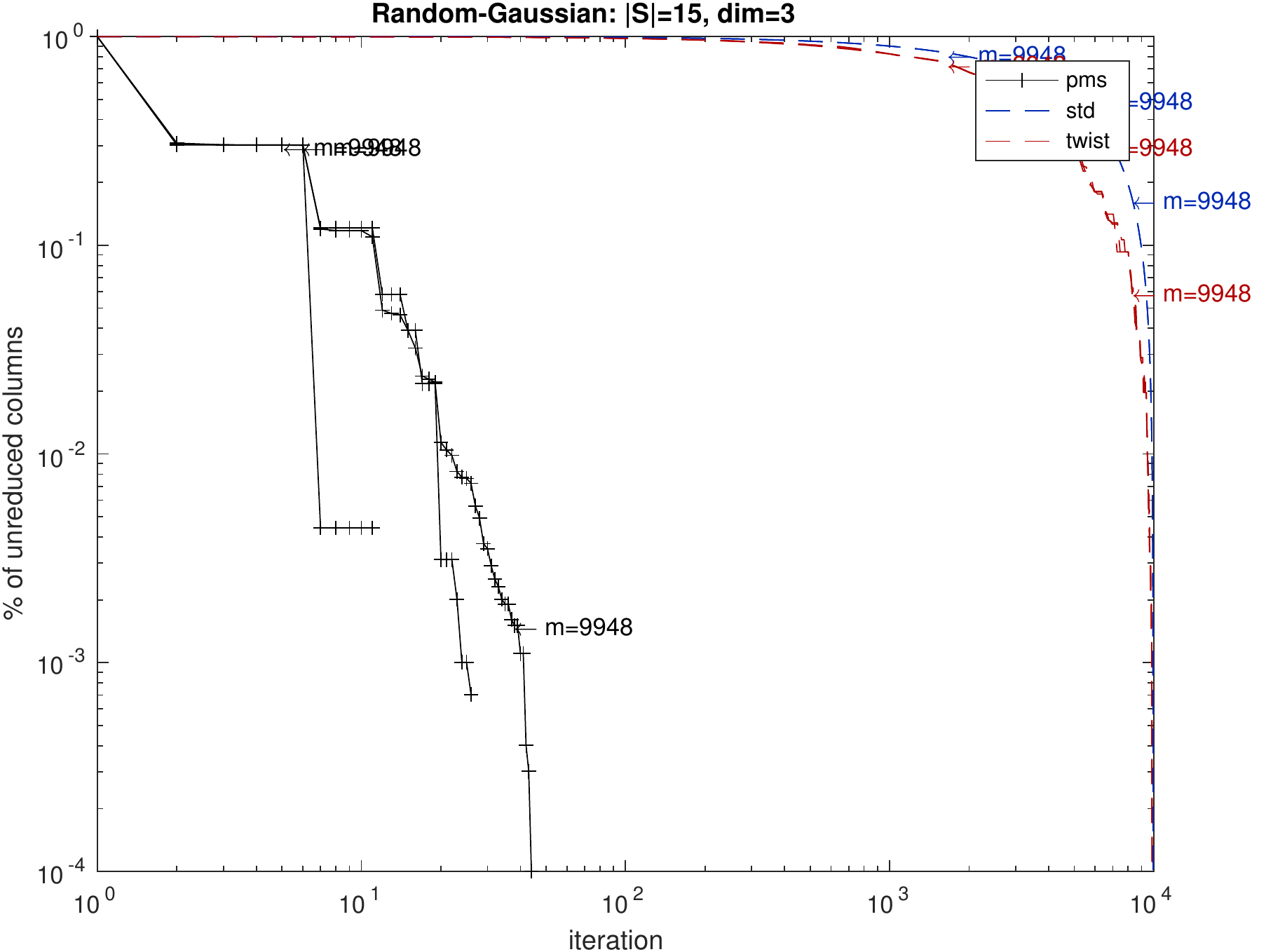}\label{fig:percentage_unreduced_gaussian}}
\subfloat[Figure-8]{\includegraphics[width=0.50\linewidth]{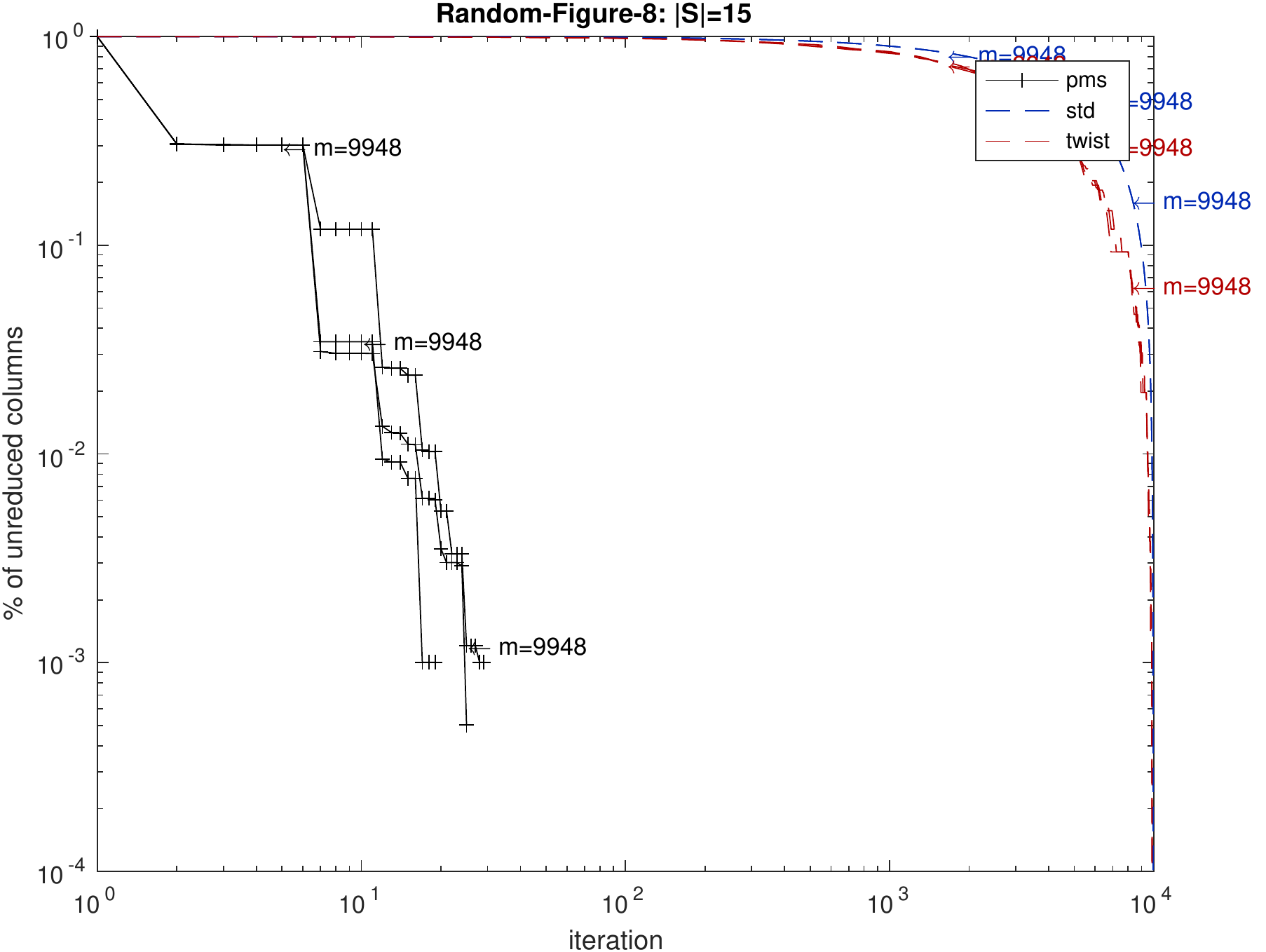}\label{fig:percentage_unreduced_f8}}\\
\subfloat[Trefoil Knot]{\includegraphics[width=0.50\linewidth]{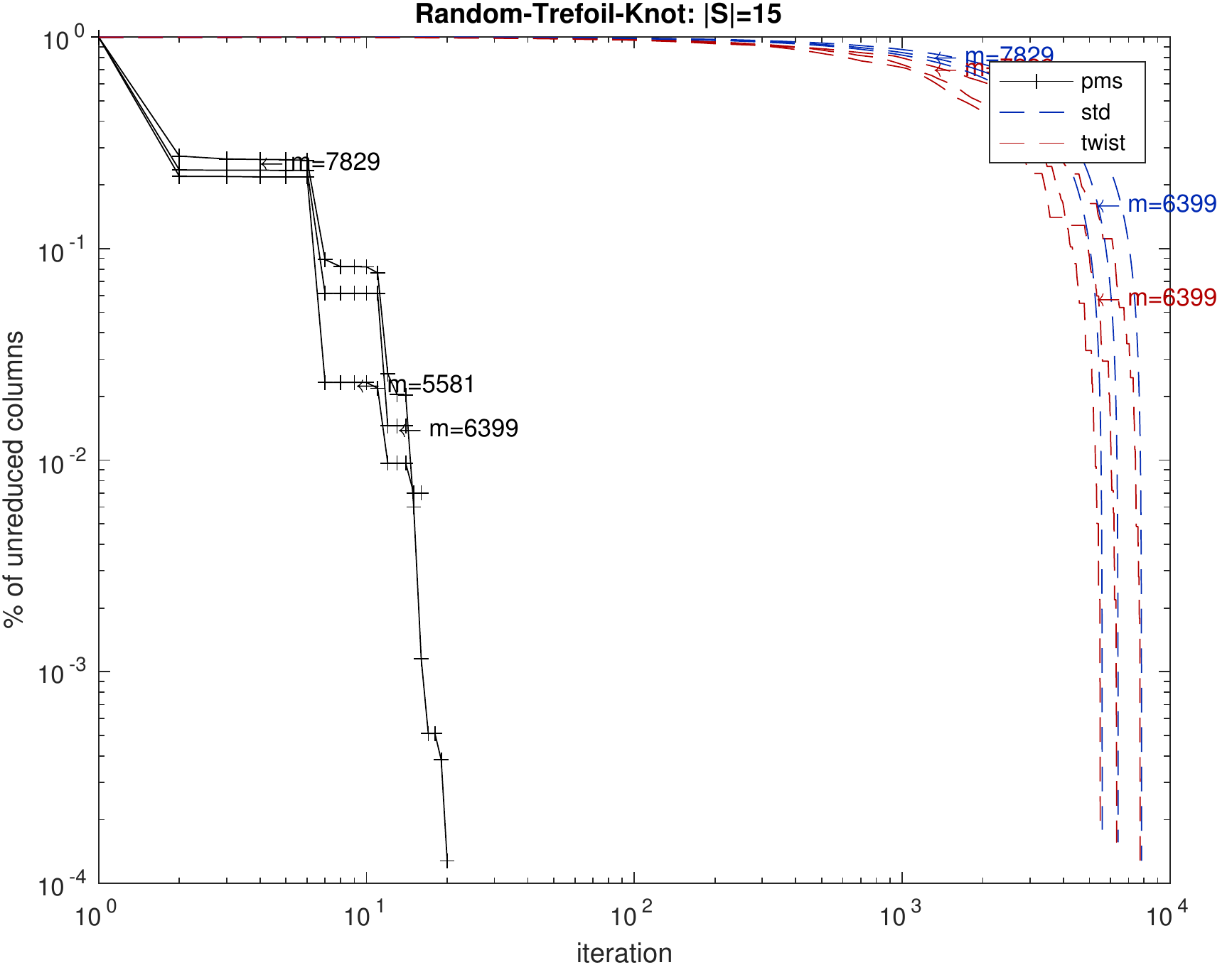}\label{fig:percentage_unreduced_trefoil}}
\subfloat[Sphere product]{\includegraphics[width=0.50\linewidth]{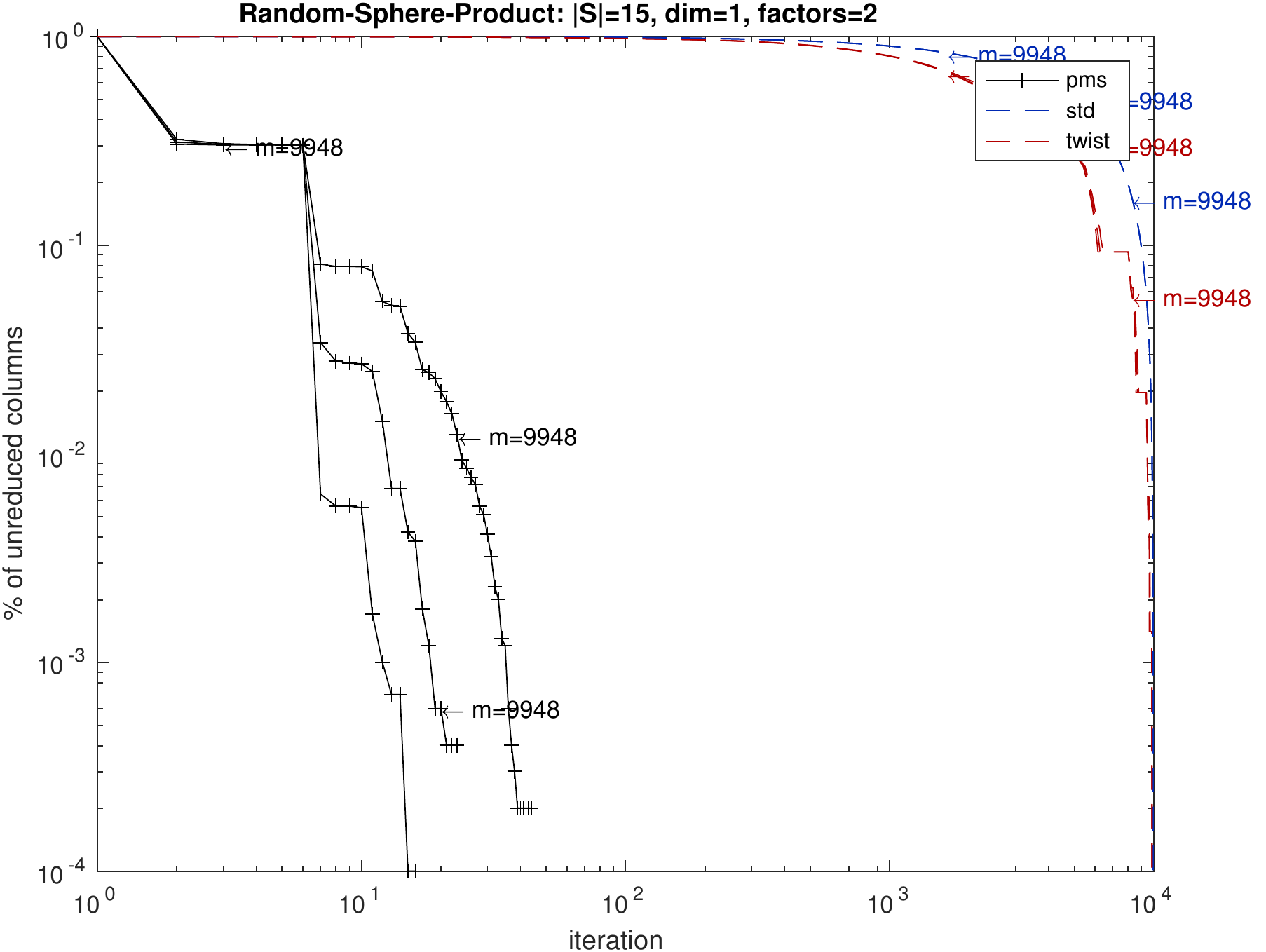}\label{fig:percentage_unreduced_sphere}}
\caption[Proportion of unreduced columns per iteration]{{\bf
    Proportion of unreduced columns per iteration.} For each ensemble,
  three point clouds are sampled and the corresponding simplicial
  complexes are reduced with each of Algs.\ \ref{alg:mat_red} -
  \ref{alg:twist}. The proportion of unreduced columns is presented at
  each iteration.} 
	\label{fig:percentage_unreduced}
\end{figure*}

\begin{small}
\begin{table*}
\centering
\begin{tabular}{l || ccc | ccc | ccc | ccc }
\toprule
\multirow{2}{*}{Proportion} &
\multicolumn{3}{c}{Gaussian}&
\multicolumn{3}{c}{Figure-8}&
\multicolumn{3}{c}{Trefoil-knot}&
\multicolumn{3}{c}{Sphere-product}\\
& {pms} & {std} & {twist} & {pms} & {std} & {twist} & {pms} & {std} & {twist} & {pms} & {std} & {twist} \\
\midrule
0.90 & 2 & 996 & 528 & 2 & 996 & 499 & 2 & 784 & 392 & 2 & 996 & 519\\
0.50 & 2 & 4975 & 3405 & 2 & 4975 & 2974 & 2 & 3916 & 2536 & 2 & 4975 & 2958\\
0.10 & 12 & 8955 & 7424 & 7 & 8955 & 6849 & 7 & 7048 & 6140 & 7 & 8955 & 6115\\
0.05 & 15 & 9452 & 8432 & 7 & 9452 & 8618 & 12 & 7439 & 6738 & 7 & 9452 & 8404\\
0.01 & 20 & 9850 & 9471 & 17 & 9850 & 9494 & 15 & 7752 & 7427 & 7 & 9850 & 9461\\
\bottomrule
\end{tabular}
\caption{{\bf Iterations to unreduced percentage.} For each ensemble,
  one point cloud is sampled and the corresponding simplicial complex
  is reduced with each of Algs.\ \ref{alg:mat_red} -
  \ref{alg:twist}. The number of iterations to achieve a given {\em
    proportion of unreduced columns} is presented for each algorithm.} 
\label{tab:iterations_unreduced}
\end{table*}
\end{small}

\subsection{The set $\essential$ can be reliably estimated after a few iterations}
\label{sec:essential_convergence}

Finally, we show how the efficacy of Lemma
\ref{lemma:essential_estimation} in estimating the the set of
essential columns.
%iteratively refine a set $E\subset [m]$ as given by the right hand side of \eqref{eq:essential_estimation} into $\essential \subset E$.
%
Let $E^{\ell}$ be defined as in \eqref{eq:essential_iteration}, by Lemma \ref{lemma:essential_estimation}, $\essential \subset E^{\ell}$ for every $\ell$.
If $E^{\ell}$ is our estimation at iteration $\ell$, then the number of true positives, false positives and false negatives in the estimation is, respectively,
\begin{align}
\mbox{TP} &= |\essential \cap E^{\ell}|  = |\essential|, \nonumber\\
\mbox{FP} &= |E^{\ell}\setminus \essential| = |E^{\ell}| - |\essential|, \nonumber\\
\mbox{FN} &= |\emptyset| = 0  \nonumber.
\end{align}
To evaluate the quality of the estimation, we compute the {\em
  precision} as defined by
\begin{align*}
\mbox{precision} &= \frac{\mbox{TP}}{\mbox{TP} + \mbox{FP}} = \frac{|\essential|}{|E^{\ell}|} \nonumber
%\mbox{recall} &= \frac{\mbox{TP}}{\mbox{TP} + \mbox{FN}} = 1 \nonumber
\end{align*}
and note that the {\em recall} $= \frac{\mbox{TP}}{\mbox{TP} +
  \mbox{FN}}$ is equal to 1 due to the estimate giving no false negatives.
Fig.\ \ref{fig:ee_precision} and Tab.\ \ref{tab:iterations_essential} show
the remarkably few iterations 
needed by Alg.\ \ref{alg:alpha_beta} to achieve precision near one
while Algs.\ \ref{alg:mat_red} and \ref{alg:twist} show increase in the
precision to one only in the later iterations.  Specifically, Alg.\
\ref{alg:alpha_beta} achieves precision of $95\%$ with two iterations
and complete precision within 8 iterations whereas Algs.\
\ref{alg:mat_red} and \ref{alg:twist} require between seven and ten
thousand iterations for similar precisions.
%increase in the precision of the estimation for each of the algorithms under consideration.
%
%We can see that Algorithm \ref{alg:alpha_beta} quickly reaches a precision that is extremely close to one while the other algorithm takes much longer due to the sequential nature.

\begin{figure*}[!htbp]%[H]
	\centering
%\subfloat[Icosahedron]{\includegraphics[width=0.33\linewidth]{{benchmark_alpha_beta_parallel-icosahedron-essential_estimation_precision}}\label{fig:ee_icoshedron}}
%\subfloat[Morozov order 7]{\includegraphics[width=0.33\linewidth]{{benchmark_alpha_beta_parallel-morozov-essential_estimation_precision}}\label{fig:ee_morozov}}
\subfloat[Gaussian]{\includegraphics[width=0.50\linewidth]{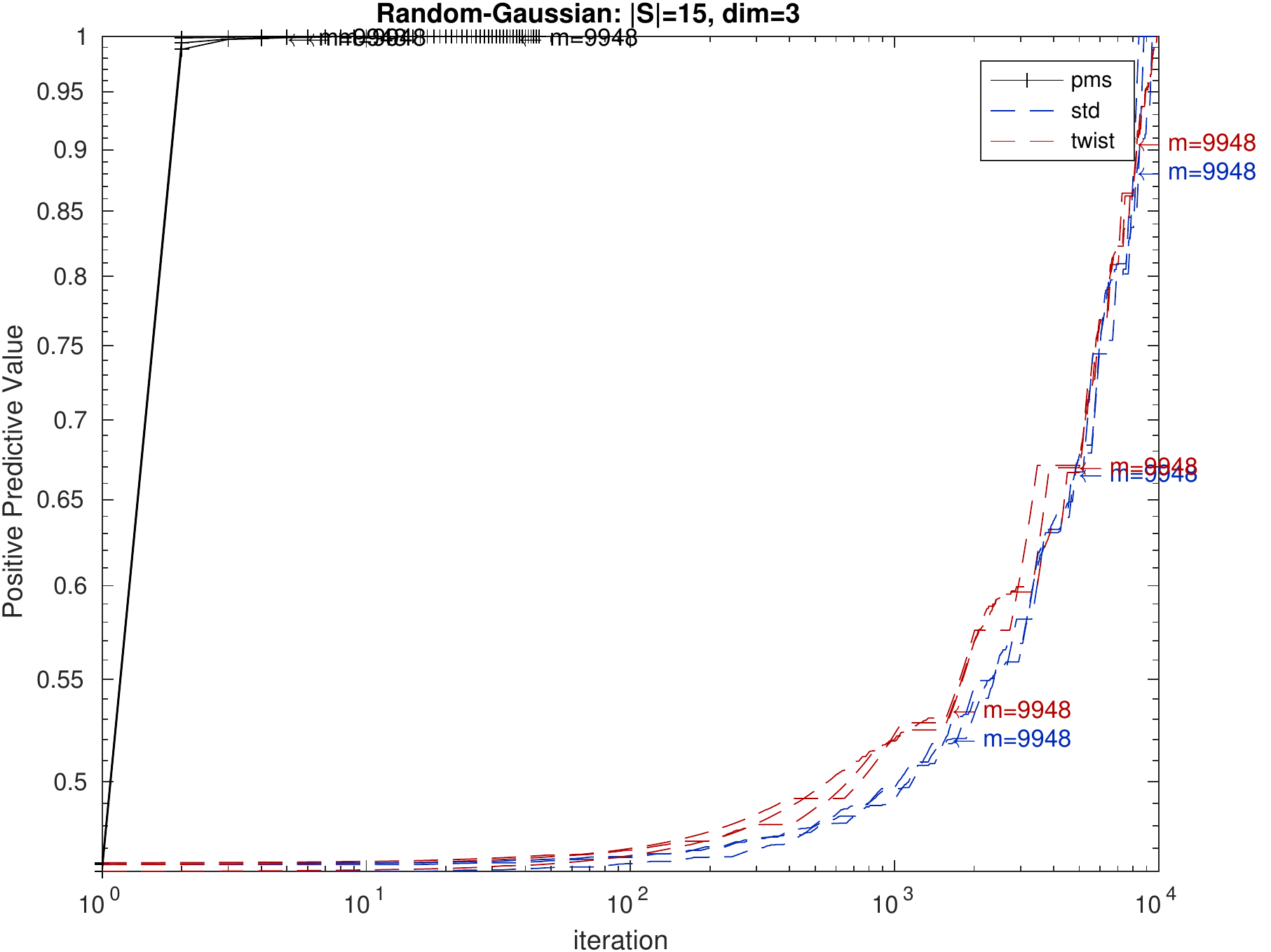}\label{fig:ee_gaussian}}
\subfloat[Figure-8]{\includegraphics[width=0.50\linewidth]{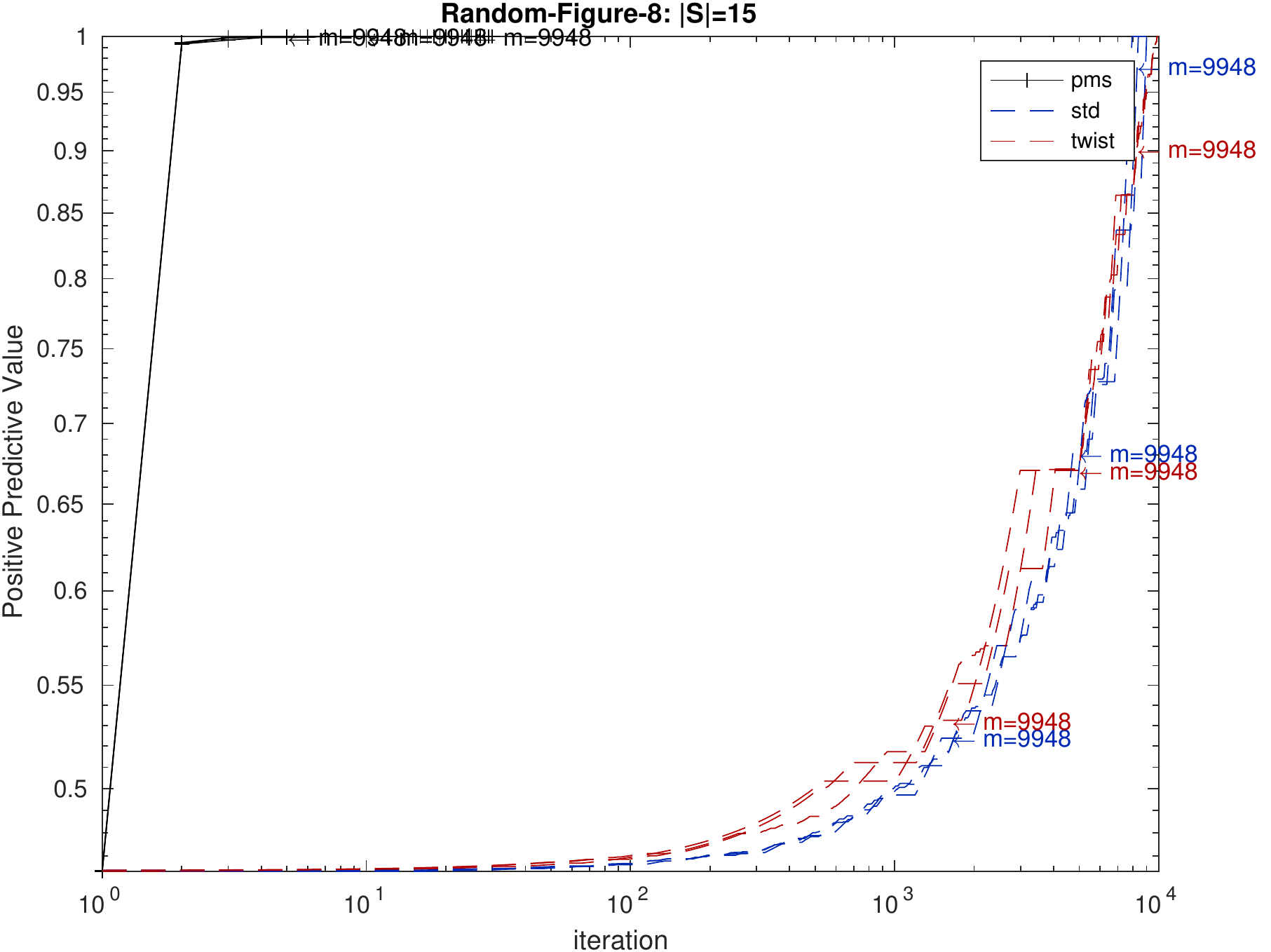}\label{fig:ee_f8}}\\
\subfloat[Trefoil Knot]{\includegraphics[width=0.50\linewidth]{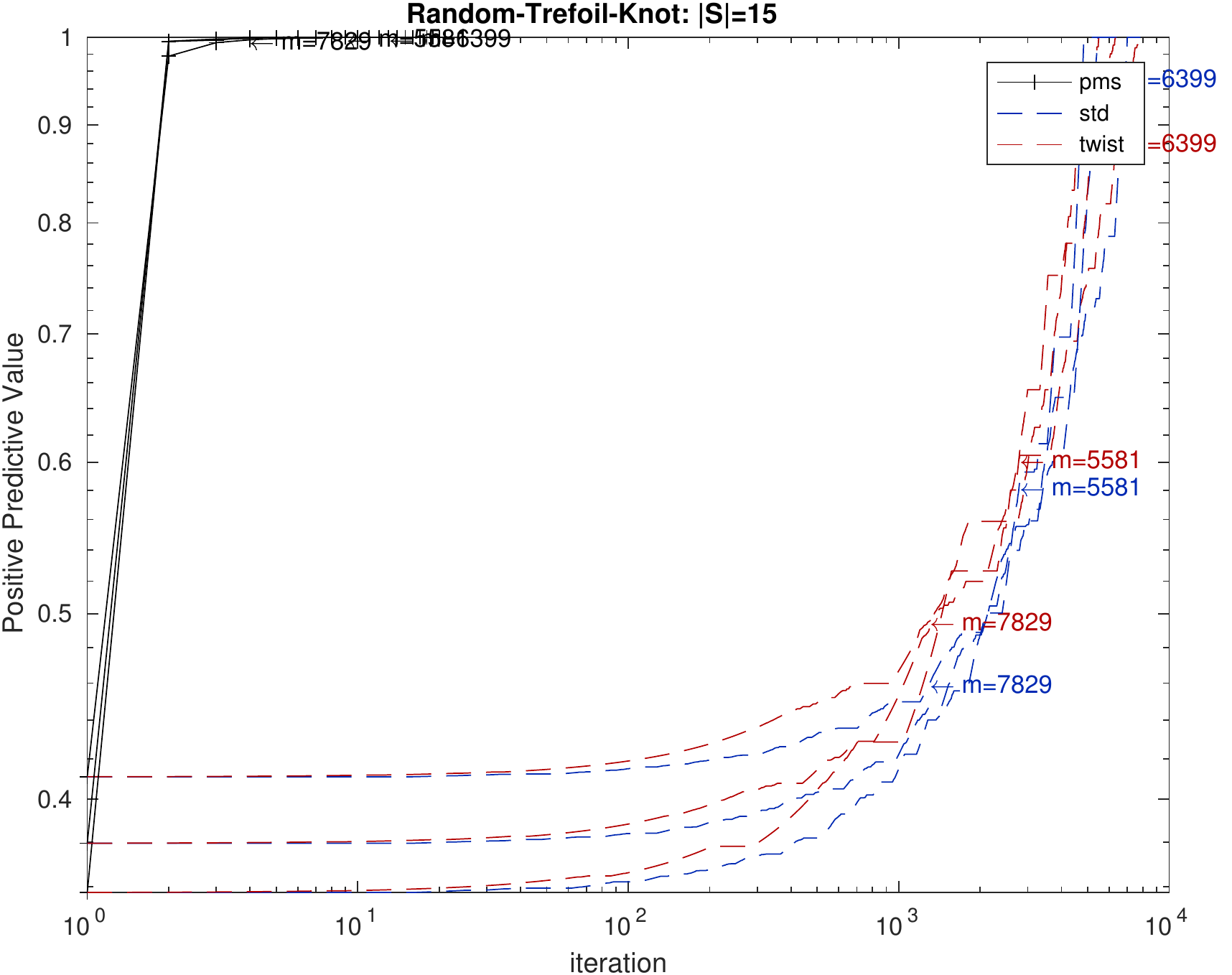}\label{fig:ee_trefoil}}
\subfloat[Sphere product]{\includegraphics[width=0.50\linewidth]{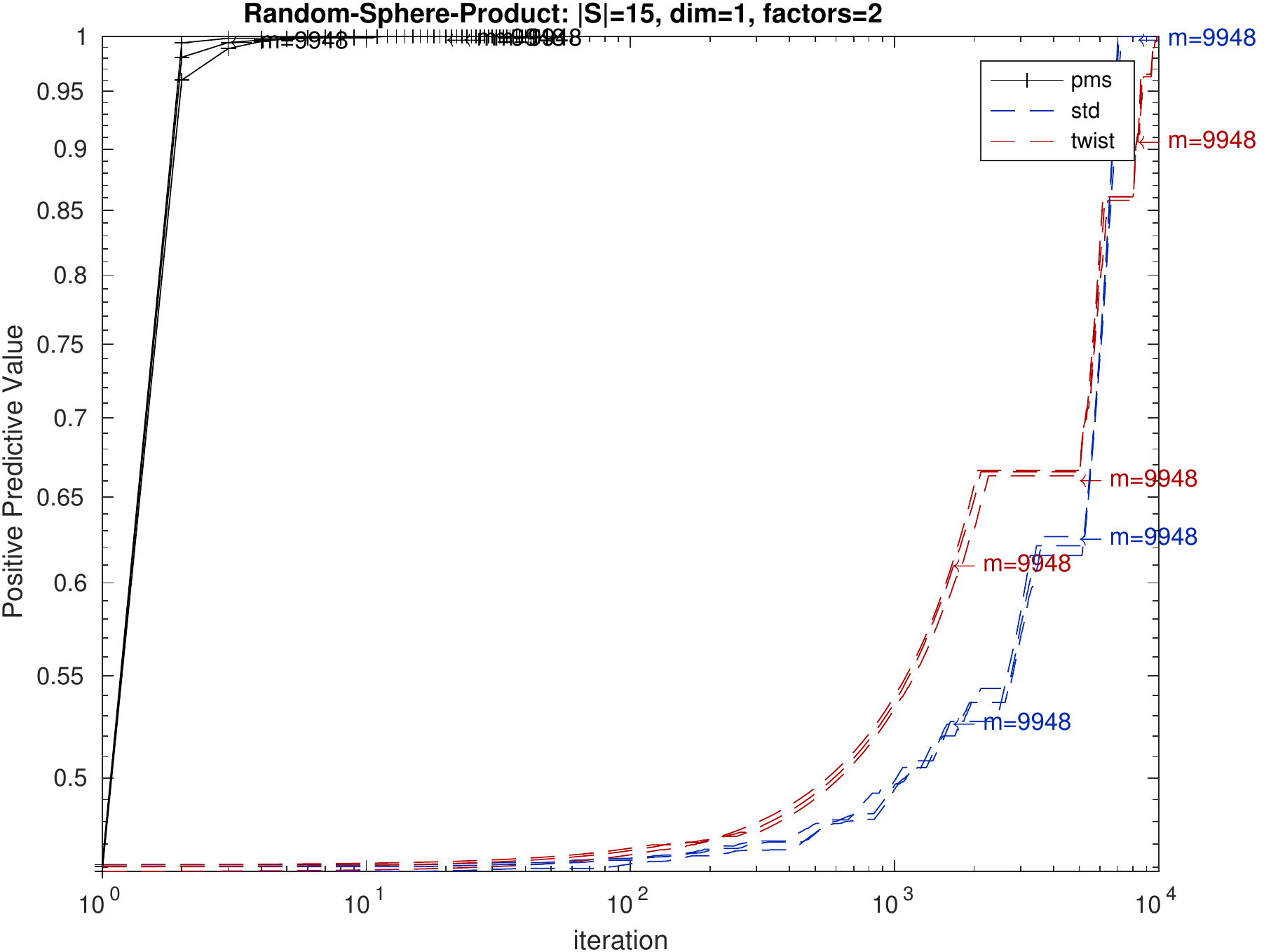}\label{fig:ee_sphere}}
\caption[Precision of $\essential$ estimation]{{\bf Precision of $\essential$ estimation.} For each ensemble, three point clouds are sampled and the corresponding simplicial complexes are reduced with each algorithm. At each iteration, $\essential$ is estimated and its precision is tracked.}
	\label{fig:ee_precision}
\end{figure*}

\begin{small}
\begin{table*}
\centering
\begin{tabular}{l|| ccc | ccc | ccc | ccc }
\toprule
\multirow{2}{*}{Precision} &
\multicolumn{3}{c}{Gaussian}&
\multicolumn{3}{c}{Figure-8}&
\multicolumn{3}{c}{Trefoil-knot}&
\multicolumn{3}{c}{Sphere-product}\\
& {pms} & {std} & {twist} & {pms} & {std} & {twist} & {pms} & {std} & {twist} & {pms} & {std} & {twist} \\
\midrule
0.10 & 1 & 1 & 1 & 1 & 1 & 1 & 1 & 1 & 1 & 1 & 1 & 1\\
0.50 & 2 & 1033 & 559 & 2 & 1231 & 493 & 2 & 2132 & 1374 & 2 & 1016 & 557\\
0.90 & 2 & 8461 & 8244 & 2 & 7599 & 8238 & 2 & 6691 & 6445 & 2 & 6675 & 8222\\
0.95 & 2 & 8636 & 8938 & 2 & 7774 & 8892 & 2 & 6871 & 7123 & 2 & 6850 & 8519\\
1.00 & 8 & 8794 & 9858 & 5 & 7932 & 9869 & 8 & 7004 & 7732 & 7 & 7008 & 9866\\
\bottomrule
\end{tabular}
\caption{{\bf Iterations to essential-estimation precision.} For each ensemble, one point cloud is sampled and the corresponding simplicial complex is reduced with each algorithm. The number of iterations to achieve a given {\em precision} of the set $\essential$ is given for each algorithm.}
\label{tab:iterations_essential}
\end{table*}
\end{small}

\section{Conclusions}

We have presented a massively parallel algorithm, Alg.\ \ref{alg:alpha_beta} for the reduction of boundary matrices in the scalable computation of persistent homology.   This work extends the foundational algorithms \cite{bauer2014clear, chen2011persistent,edelsbrunner2002topological} using many of the same notions, but allowing a dramatically greater distribution of the necessary operations. 
Our numerical experiments show that Alg\ \ref{alg:alpha_beta}, as compared with Algs.\ \ref{alg:mat_red} and \ref{alg:twist}, is able to pack more operations into few iterations, approximates $\lowstar$ simultaneously at all scales of the simplicial complex filtration, and determines the essential columns in remarkably few iterations.
%
%Our algorithm leverages the structure of the boundary matrix in the sense that it minimises left-to-right column operations and multi-threads the operations on independent batches that do not necessarilty reduce the columns, thus profiting from the clearing strategy.
%
%Moreover, further numerical experiments show that Alg.\ \ref{alg:alpha_beta} outperforms the sequential Algs.\ \ref{alg:mat_red} and \ref{alg:twist} on achieving approximmations to $\lowstar$ and on finding an estimation of essential columns.
This massively parallel algorithm suggests the reduction of dramatically larger boundary matrices will now be possible, and moreover allows early termination with accurate results when computational constraints are reached.
Implementation of Alg.\ \ref{alg:alpha_beta} in the leading software packages, as reported in \cite{otter2015roadmap}, is underway and we expect to report dramatic reduction in computational times in a subsequent manuscript.

\section{Acknowledgements}

The authors would like to thank Vidit Nanda for his helpful comments.
This work was supported by The Alan Turing Institute under the EPSRC grant EP/N510129/1.
RMS acknowledges the support of CONACyT.

%\input{appendix/basis}
%\input{appendix/morozov}
%\input{appendix/c8}
%\input{appendix/cech}
%\input{appendix/benchmarking}

%%%%%%%%%%%%%%%%%%%%%%%%%%%%%%%%%%%%%%%%%%%%%%%%%%%%%%%%%%%%%%%%%%%%%%%%%%%

\bibliographystyle{plain}
\bibliography{ph}

\end{document}